\documentclass[letterpaper,reqno]{amsart}
\usepackage{graphicx}
\usepackage{amsmath,cite,amsthm,thmtools}
\usepackage{amsfonts}
\usepackage[margin=1in]{geometry}
\usepackage{amssymb}
\usepackage{mathrsfs}
\usepackage{blkarray}
\usepackage{comment}
\usepackage[dvipsnames]{xcolor}
\usepackage{esvect}
\usepackage{bm}

\usepackage[colorlinks=true, pdfstartview=FitV, linkcolor=blue, citecolor=blue, urlcolor=blue]{hyperref}
\usepackage{url, hypcap}
\hypersetup{colorlinks=true, citecolor=darkblue, linkcolor=darkblue}
\definecolor{darkblue}{rgb}{0.0,0,0.7} 

\usepackage{tikz}
\usepackage{tikz-cd}
\usetikzlibrary{graphs,patterns,decorations.markings,arrows,matrix,automata,snakes}
\tikzset{
	edge/.style={->,> = latex'}
}
\usetikzlibrary{calc,through,backgrounds,shapes}
\usetikzlibrary{positioning}
\usetikzlibrary{decorations.pathreplacing,angles,quotes}

\usepackage{bbm}

\usepackage[colorinlistoftodos]{todonotes}

\tikzset{
	edge/.style={->,> = latex'}
}

\newcommand{\T}{\mathcal{T}_{\symm_n}}
\newcommand{\Tq}{\mathcal{T}_{\symm_n}(q)}
\newcommand{\Tqx}{\mathcal{T}_{\symm_n}(q,{\sf x})}
\newcommand{\Tx}{\mathcal{T}_{\symm_n}({\sf x})}

\newcommand{\X}{\mathcal{X}_{\symm_n}}

\newcommand{\Psiqx}{\Psi(q,\x)}
\newcommand{\Psix}{\Psi(\x)}
\newcommand{\x}{{\sf x}}
\newcommand{\y}{{\sf y}}
\newcommand{\xo}{\overline{x}}
\newcommand{\yo}{\overline{y}}
\newcommand{\xbar}{\overline{{\sf x}}}
\newcommand{\ybar}{\overline{{\sf y}}}

\renewcommand{\P}{\mathbb{P}}
\newcommand{\E}{\mathbb{E}}

\newcommand{\tTq}{\widetilde{\mathcal{T}}_{G/B}(q)}
\newcommand{\tTqx}{\widetilde{\mathcal{T}}_{G/B}(q,{\sf x})}

\newcommand{\hatTqx}{\mathcal{T}_{G/B}(q,\x)}
\newcommand{\hatTq}{\mathcal{T}_{G/B}(q)}

\newcommand{\m}{{\sf m }}
\newcommand{\bv}{{\sf b }}

\newcommand{\av}{{\sf a }}

\newcommand{\wordT}{\mathcal{T}_{W_{\m}}}
\newcommand{\wordTx}{\mathcal{T}_{W_{\m}}(\xbar)}

\newcommand{\wordTqx}{\mathcal{T}_{W_{\m}}(q,\xbar)}
\newcommand{\wordX}{\mathcal{X}_{W_{\m}}}

\newcommand{\bB}{B^-}

\newcommand{\std}{{\sf std}}
\newcommand{\destd}{{\sf destd}}

\newcommand{\HH}{\mathcal{H}}

\newcommand{\symm}{\mathfrak{S}}
\newcommand{\C}{\mathbb{C}}
\newcommand{\F}{\mathbb{F}}

\newcommand{\GL}{\mathsf{GL}}

\newtheorem{theorem}{Theorem}[section]
\newtheorem{lemma}[theorem]{Lemma}
\newtheorem{prop}[theorem]{Proposition}
\newtheorem{cor}[theorem]{Corollary}
\theoremstyle{definition}
\newtheorem{defn}[theorem]{Definition}
\newtheorem{example}[theorem]{Example}
\newtheorem{remark}[theorem]{Remark}
\numberwithin{equation}{section}
\usepackage[capitalize,nameinlink,noabbrev,nosort]{cleveref}
\crefname{lemma}{Lemma}{Lemmas}
\crefname{cor}{Corollary}{Corollaries}
\crefname{defn}{Definition}{Definitions}
\crefname{prop}{Proposition}{Propositions}

\DeclareMathOperator{\inv}{inv}
\DeclareMathOperator{\proj}{proj}
\DeclareMathOperator{\incl}{incl}
\DeclareMathOperator{\projs}{p}
\DeclareMathOperator{\incls}{\iota}

\DeclareMathOperator{\coinv}{coinv}
\DeclareMathOperator{\LRM}{LRM}

\title{$q$-deformations of the Tsetlin library}

\author[A.\ Ayyer]{Arvind Ayyer}
\address[A.\ Ayyer]{Department of Mathematics, Indian Institute of Science, Bangalore 560012, India}
\email{arvind@iisc.ac.in}
\urladdr{\href{https://math.iisc.ac.in/~arvind/}{https://math.iisc.ac.in/~arvind/}}

\author[S.\ Brauner]{Sarah Brauner}
\address[S.\ Brauner]{Division of Applied Mathematics, Brown University, Providence, RI, USA}
\email{sarahbrauner@gmail.com}
\urladdr{\href{https://www.sarahbrauner.com/}{https://www.sarahbrauner.com/}}

\author[J.\ de Gier]{Jan de Gier}
\address[J.\ de Gier]{School of Mathematics and Statistics, University of Melbourne, Victoria 3010, Australia}
\email{jdgier@unimelb.edu.au}
\urladdr{\href{https://blogs.unimelb.edu.au/jan-de-gier/}{https://blogs.unimelb.edu.au/jan-de-gier/}}

\author[A.\ Schilling]{Anne Schilling}
\address[A.\ Schilling]{Department of Mathematics, University of California, One Shields
Avenue, Davis, CA 95616-8633, U.S.A.}
\email{aschilling@ucdavis.edu}
\urladdr{\href{http://www.math.ucdavis.edu/~anne}{http://www.math.ucdavis.edu/~anne}}

\begin{document}

\begin{abstract}
The Tsetlin library is a random shuffling process on permutations of $n$ letters, where each letter $i$ can be 
interpreted as a book; book $i$ is brought to the front of the bookshelf with an assigned probability $x_i$. 
We define a $q$-deformation of the Tsetlin library by replacing the symmetric group action on permutations 
by the action of the type $A$ Iwahori-Hecke algebra. We compute the stationary distribution and 
spectrum of this Markov chain by relating it to a Markov chain on complete flags over the finite field vector 
space $\mathbb{F}_q^n$ and applying techniques from semigroup theory. 
We prove that for a natural choice of $x_i$ the total variation distance mixing time of the $q$-Tsetlin library on 
permutations of $n$ is $O(n)$ compared to $\Theta(n \log n)$ for the Tsetlin library at $q=1$, which demonstrates a phase transition.
We also generalize the $q$-Tsetlin library to words (with repeated letters), and compute its stationary distribution and spectrum.
\end{abstract}


\maketitle
\tableofcontents

\section{Introduction}

In recent years, the combinatorial theory of Markov chains has emerged as an important area with deep connections to fields 
beyond combinatorics and probability, including representation theory~\cite{diaconis-1988} and geometry~\cite{BD.1998,DiaconisMorton.2025}. 
A central class of examples comes from card-shuffling~\cite{DiaconisFulman.2023}, that is, Markov chains on the symmetric group.

A classical but simplistic shuffling algorithm is the \textit{random-to-top shuffle}, in which a randomly chosen card is moved to the top 
of the deck at each step with uniform probability.
The \emph{Tsetlin library}~\cite{Tsetlin.1963} generalizes this shuffle by letting the probability of the card being picked depend on the card itself. 

The process can be described as follows. 
Imagine $n$ books arranged linearly on a shelf. At each step, a book is selected according to a probability distribution reflecting its popularity, 
checked out, and then returned to the front of the shelf. Over time, the most popular books tend to accumulate near the front of the collection,
yielding the stationary distribution.

More than sixty years after its introduction, the Tsetlin library continues to inspire mathematical and scientific innovation.
For instance, this process has numerous applications in fields such as computer
science~\cite{BSTW.1986,Donnelly.1991,Fill.1996a}, where it is sometimes known as \emph{dynamic file maintenance} or \emph{cache maintenance}, 
as well as genetics, where its stationary distribution is known as the GEM (Griffiths--Engen--McCloskey) distribution \cite{Donnelly.1991}. 
More recently, it has spurred far-reaching generalizations to structures such as hyperplane arrangements~\cite{bidigare_hanlon_rockmore.1999,BD.1998}, 
linear extensions of posets~\cite{AyyerKleeSchilling.2014}, and many more, see e.g. \cite{CHATTERJEE2024139,Fill.1996} and the references therein. 
The study of the Tsetlin library has also led to an understanding of the 
stationary distribution of Markov chains arising from $\mathscr{R}$-trivial monoids~\cite{ASST.2015},
and to general finite state Markov chains~\cite{RhodesSchilling.2019}.
It has further sparked new methods for computing the rate of convergence of Markov chains 
to the stationary distribution, starting with the work of Brown and Diaconis~\cite{BD.1998,Brown}
and Bidigare--Hanlon--Rockmore~\cite{bidigare_hanlon_rockmore.1999},
and generalized to arbitrary finite state Markov chains by Rhodes and the fourth author~\cite{RhodesSchilling.2022}.

In this paper, we define several novel $q$-deformations of the Tsetlin library, analyze their stationary distributions and spectra, and 
compute their mixing times. 
We do so by lifting the Tsetlin library to the \emph{type $A$ Iwahori Hecke algebra} $\HH_n(q)$.
The parameter $q$ can be understood as encoding a probability, thereby enriching the original Tsetlin library. We show that for a natural choice of 
probabilities and weights, the total variation distance mixing time of our $q$-Tsetlin library on permutations of $n$ for $q>1$ is $O(n)$,
compared to the known results $\Theta(n \log n)$ for the Tsetlin library at $q=1$~\cite{diaconisICM,Diaconis.1996,Fill.1996,Nestoridi.2019}. 
Our work contributes to the growing body of work on Markov chains built using the Hecke 
algebra \cite{qr2r,qr2r2,diaconis2000analysis}, which can be used to model many interacting particle systems, including the asymmetric simple exclusion 
process (ASEP), see e.g. \cite{Cantini_2015,10.1093/imrn/rny159,bufetov-2020}, as well as Markov chains on $\GL_n(\F_q)$,  
see e.g.~\cite{diaconis2023double,DiaconisMorton.2025}.
We note in passing that recently another model called the $p$-Tsetlin library has been studied on the group of colored permutations 
with $p$ colors~\cite{nakagawa-nakano-2025}, and is unrelated to this work.

We begin with a formal definition of the Tsetlin library.

\subsection{Definition of the Tsetlin library} 
The Tsetlin library can be defined formally using the symmetric group $\symm_n$ and its corresponding group algebra $\C \symm_n$. Recall that 
$\symm_n$ is generated by adjacent transpositions $s_i$ that interchange the letters $i$ and $i+1$. We first consider the \emph{random-to-top shuffle}, 
defined as the element 
\begin{equation} \label{eq:defR2Tintro}
	\T := \frac1n \sum_{i=1}^{n} s_{i-1}\cdots s_1, 
\end{equation}
that acts on a permutation $\pi \in \symm_n$ by right multiplication (i.e. on the positions of the letters in $\pi$); by convention, we interpret the $i=1$ 
term in the summand as the identity. Specifically, writing $\pi=\pi_1 \,\ldots \,\pi_n \in \symm_n$ in one-line notation, the action of $\T$ on $\pi$ is given by
\begin{equation}
\label{eq:R2T}
	\pi \cdot \T = \frac1n \sum_{i=1}^n \pi_i \, \pi_1 \ldots \, \hat{\pi}_i \, \ldots \,\pi_n,
\end{equation}
where the hat indicates that the letter is omitted.
In other words, the operator $\T$ moves the letters of $\pi$ to the front in all possible ways. 

The coefficient of each term $\sigma$ in the right hand side of \eqref{eq:R2T} is the probability of obtaining $\sigma$ from $\pi$ 
after one iteration of $\T$.
The \emph{Tsetlin library} is defined by introducing probabilities $\x = (x_1, \ldots, x_n)$ to each $i \in [n]:=\{ 1, \ldots, n\}$ 
satisfying $x_1+\cdots+x_n=1$ via an operator $\X(\x)$ 
that acts on $\pi$ by
\begin{equation}
	\pi\cdot \X(\x) := x_{\pi_1} \pi.
\label{eq:X}
\end{equation}
The Tsetlin library $\Tx$ is then defined as the composition (acting on the right)
\begin{equation}
\label{equation.Tx}
	\Tx= \T \circ \X(\x) ,
\end{equation}
which acts on $\pi$ as
\begin{equation}
\pi \cdot \Tx = \sum_{i=1}^n x_{\pi_i} (\pi_i \, \pi_1 \ldots \, \hat{\pi}_i \, \ldots \, \pi_n).
\end{equation}
In other words, $\Tx$ brings the letter $\pi_i$ to the front of the permutation with probability $x_{\pi_i}$. For example,
\[ 
	123 \cdot \Tx = x_{1} 123 + x_2 213 + x_3 312. 
\]
Viewing $\x$ as a set of formal variables, this action can be extended by linearity to an action on 
$\mathbb{C}(\x) \symm_n$, where $\mathbb{C}(\x) = \mathbb{C}(x_1,\ldots,x_n)$ 
is the fraction field of rational functions in the variables $x_1,\ldots,x_n$.

The Tsetlin library is very well known in part because of the elegance of its spectral properties. 
The \emph{stationary distribution} $\Psix \in \mathbb{C}(\x) \symm_n$ is the left eigenvector of $\Tx$
with eigenvalue $x_1+\cdots+x_n$. When we interpret the $x_i$ as probabilities satisfying the condition $0\leqslant x_i \leqslant 1$ with
$x_1+\cdots+x_n=1$, the stationary state is the left eigenvector of eigenvalue one, that is, $\Psix \cdot \Tx = \Psix$.

Writing $\Psix = \sum_\pi \pi\Psi(\x)_\pi$, Tsetlin \cite{Tsetlin.1963} proved (see also \cite{Hendricks.1972, Donnelly.1991}) that the components of the stationary distribution are given by
\begin{equation}
\label{equation.Psi}
	\Psi(\x)_\pi = \prod_{i=1}^n \frac{x_{\pi_i}}{1-\sum_{j=1}^{i-1}x_{\pi_j}}.
\end{equation}
 
The eigenvalues of $\Tx$ also admit a simple and beautiful description. In particular, independent work of Donnelly \cite{Donnelly.1991}, 
Kapoor--Reingold \cite{kapoor1991stochastic}, and Phatarfod \cite{Phatarfod.1991} show that the eigenvalues of $\Tx$ are of the form 
\begin{equation} \lambda_S(\x):= \sum_{i \in S} x_i, \label{eq:tsetlineigen}
\end{equation} 
for every subset $S \subseteq [n].$ Moreover, the multiplicity of $\lambda_S(\x)$ is given by $d_{n-|S|}$, where 
\begin{equation}
\label{derange}
d_n:= n! \sum_{k=0}^n \frac{(-1)^k}{k!} 
\end{equation}
is the $n$-th \emph{derangement number} counting permutations of $\symm_n$ with no fixed points. The \textit{mixing time} of the Tsetlin library, which measures how quickly the Markov chain converges to stationarity, was computed to be
$\Theta(n \log n)$~\cite{diaconisICM,Fill.1996,Nestoridi.2019}.

In what follows, we will obtain generalizations of each of equations \eqref{equation.Psi} and \eqref{eq:tsetlineigen}, and recover the multiplicity 
count in~\eqref{derange}. 

\subsection{$q$-deformations of the Tsetlin library}
Here, we define several $q$-deformations of the Tsetlin library by generalizing the process in \eqref{equation.Tx} to incorporate a probability 
$q^{-1} \in (0,1]$ and defining its action on
\begin{enumerate}
    \item the module over flags $\C[G/B]$, where $G = \GL_n(\F_q)$ for $\F_q$ a finite field, and $B$ is its subgroup of upper triangular matrices;
    \item the symmetric group algebra $\C\symm_n$; and 
    \item the space of words $\C W_\m$ of content $\m$, where $\m = (m_1, \ldots, m_\ell)$ is a composition of $n$.
\end{enumerate}
In all cases, we compute the stationary distribution and characteristic polynomial obtained from this process, generalizing \eqref{equation.Psi} 
and~\eqref{eq:tsetlineigen}. We write $\hatTqx$, $\Tqx$ and $\wordTqx$ to refer to the $q$-Tsetlin library acting on the flag variety, the symmetric 
group algebra, and the space of words, respectively. Here, $\xbar$ is a transformation of the probabilities $\x$, defined in \eqref{equation.rates xbar x}.

Our $q$-deformation of $\Tx$ is achieved by working in the \emph{type $A$ Iwahori-Hecke algebra} $\HH_n(q)$ and defining an appropriate 
operator $\mathcal{X}$ generalizing \eqref{eq:X}. Our analysis builds on work of \cite{Brown,BraunerComminsReiner} which defines $q$-deformations 
of the \emph{random-to-top} operator $\T$ in \eqref{eq:defR2Tintro} in the context of the flag variety $\C[G/B]$, and \cite{qr2r,Lusztig03} who define an 
analogue of $\T$ in $\HH_n(q)$ called \emph{$q$-random-to-top} $\Tq$. See \cref{figure.notation} for a summary how our $q$-Tsetlin library 
generalizes these known cases when acting on $\C \symm_n$. \begin{figure}[!h]
\begin{center}
\begin{tikzpicture}[auto]
  \node(AT) at (0,0.7) {\emph{$q$-Tsetlin library}};
  \node (A) at (0,0) {$\Tqx$};
  \node(BT) at (6,0.7) {\emph{$q$-random-to-top}};
  \node (B) at (6,0) {$\Tq$};
  \node (C) at (0,-3) {$\Tx$};
  \node (CT) at (0,-3.7) {\emph{Tsetlin library}};
  \node (D) at (6,-3) {$\T$};
  \node (DT) at (6,-3.7) {\emph{random-to-top}};

  \draw[->] (A) -- (B) node[midway,above] {$x_i = q^{n-i}/[n]_q$};
  \draw[->] (A) -- (C) node[midway,left] {$q=1$};
  \draw[->] (B) -- (D) node[midway,right] {$q=1$};
  \draw[->] (C) -- (D) node[midway,below] {$x_i=1/n$};
\end{tikzpicture}
\end{center}
\caption{The relationship between the $q$-Tsetlin library on permutations and its specializations to $q$-random-to-top and the classical Tsetlin library. 
\label{figure.notation}}
\end{figure}

The novelty of our work is to introduce the \emph{additional probabilities} $x_1,\ldots,x_n$ to the $q$-deformation $\Tq$. This enriches the original process, 
but significantly complicates its stationary and spectral properties, as well as the methods needed to analyze them. As we will show, probabilities $x_i$ can be chosen such that the mixing time is faster than the original Tsetlin library.

\subsection{Our methods}
While our work generalizes both the $(q=1)$ Tsetlin library and $q$-random-to-top, the methods used to understand these processes 
do not immediately extend to the $q$-Tsetlin library. 

In particular, in ground-breaking work, 
Bidigare--Hanlon--Rockmore~\cite{bidigare_hanlon_rockmore.1999} and Brown \cite{Brown} showed that $\Tx$ can be realized as an element of a certain type of 
semigroup called a \emph{left regular band}, and determined a method of computing the characteristic polynomial of any Markov process arising in 
this way. Modern techniques from semigroup theory enable one to compute the stationary distribution $\Psi(\x)$ for $\Tx$ by computing the 
Karnofsky--Rhodes and McCammond expansion of the right Cayley graph of the corresponding left regular band, 
see~\cite{AyyerKleeSchilling.2014, ASST.2015, RhodesSchilling.2019}, and mixing times~\cite{RhodesSchilling.2022}. 
(These methods work more generally for  $\mathscr{R}$-trivial monoids.)

On the other hand, once one moves to $\HH_n(q)$, it is no longer obvious how the processes of interest can be formulated as a left regular band. 
In the case of $\Tq$, the analysis in \cite{qr2r} links the Markov chain with a left regular band defined in $\C[G/B]$ introduced by Brown~\cite{Brown} 
called the \emph{$q$-free left regular band} $\mathcal{F}_n(q)$. However, the techniques in \cite{qr2r} 
rely heavily on representation theory; once rates $\x$ are introduced 
in the process, ``symmetry'' is broken, and one can no longer utilize the same powerful techniques. 

To combat this obstacle, we proceed as follows. As in the case of $\Tq$, we link the $q$-Tsetlin library on flags $\hatTqx$ to the $q$-free left regular 
band $\mathcal{F}_n(q)$. Instead of using representation theory to understand $\Tqx$ and $\wordTqx$, we instead define appropriate inclusion and 
projection maps, and prove that the following diagrams commute:
\begin{figure}[h!]
\begin{tikzcd}[column sep=huge, row sep=huge]
\C(q,\x)[G/B] \arrow[d, "\projs"'] \arrow[r, "\hatTqx"] & \C(q,\x)[G/B] \arrow[d, "\projs"']   \\
\arrow[u, hook, shift right=1.5ex, "\incls"'] \C(q,\x)[\symm_n] \arrow[r, "\Tqx"] \arrow[d, "\projs"'] \arrow[u, hook, shift right=1.5ex, "\incls"'] &  \C(q,\x)[\symm_n] \arrow[d, "\projs"'] \arrow[u, hook, shift right=1.5ex, "\incls"'] \\
\C(q,\xbar)[W_{\m}] \arrow[r, "\wordTqx"] \arrow[u, hook, shift right=1.5ex, "\incls"'] &  \C(q,\xbar)[W_{\m}] \arrow[u, hook, shift right=1.5ex, "\incls"']
\end{tikzcd}
\caption{Commuting diagram relating $q$-Tsetlin libraries on flags (top), permutations (middle) and words (bottom).}
\label{figure.commuting diagram}
\end{figure}

The projection maps allow us to obtain the stationary distribution of $\Tqx$ and $\wordTqx$ by computing the stationary distribution of $\hatTqx$ 
using the left regular band methods in~\cite{ASST.2015, RhodesSchilling.2019}. Note that an explicit formula for the stationary distribution of 
$\hatTqx$ is new in and of itself. The inclusion maps then give rise to the spectra of these processes. Furthermore, we use methods developed in~\cite{RhodesSchilling.2022} to compute upper bounds on the mixing time.
  
\subsection{Outline of the paper}
The paper is organized as follows.
\begin{itemize}
    \item In \cref{section.results}, we define the three $q$-deformations of the Tsetlin library $\hatTqx, \Tqx$ and $\wordTqx$. 
    We state our main results, namely the stationary distribution, the spectrum for each of these three, and convergence to stationarity.
    \item In \cref{section.flags}, we define the inclusion and projection maps, and show that the diagram in \cref{figure.commuting diagram} 
    commutes. We also relate the $q$-Tsetlin library to the $q$-free left regular band.
    \item In \cref{section.eigenvalues}, we compute the characteristic polynomial for each case.
    \item In \cref{section.stationary distribution}, we derive the stationary distribution for each case.
    \item In \cref{section.convergence}, we analyze convergence to stationarity and proof the upper bounds on mixing times.
\end{itemize}

\subsection*{Acknowledgements}
We are extremely grateful to Darij Grinberg for communicating to us the proof of \cref{darij_lemma}, and other illuminating conversations.
We also wish to thank Pavel Galashin, Amritanshu Prasad and Arun Ram for valuable discussions. 
Furthermore, we wish to thank IPAM and ICERM for the stimulating research
atmospheres, where part of this work was done. Specifically, some of this material is based on work supported by the National Science Foundation, while 
the authors were in residence at IPAM at the semester program ``Geometry, Statistical Mechanics, and Integrability'' in Spring 2024, and
under NSF Grant No. DMS-1929284, while the authors
were in residence at the Institute for Computational and Experimental Research in Mathematics in Providence, RI, during the 
``Categorification and Computation in Algebraic Combinatorics'' semester program in Fall 2025.
Any opinions, findings, and conclusions or recommendations expressed in this material are those of the authors and do not necessarily 
reflect the views of the National Science Foundation. 

AA acknowledges support from the DST FIST Program 2021 TPN--700661. SB was partially supported by NSF MSPRF DMS--2303060.
JdG is supported by the Australian Research Council.
AS was partially supported by  NSF grant DMS--2053350 and Simons Foundation grant MPS-TSM--00007191. 

\section{Definitions and statements of main results}
\label{section.results}

In this section, we define the various $q$-deformations of the Tsetlin library and state our main results including
their stationary distributions, eigenvalues and multiplicities, and convergence to stationarity.

\subsection{Definition of the $q$-Tsetlin library}
\label{section.qTsetlin}
To define the $q$-Tsetlin library, we first define the Hecke algebra.

\begin{defn}[Hecke algebra]
The \emph{type $A$ Iwahori Hecke algebra} $\mathcal{H}_n(q)$ is the $\C(q)$-algebra generated by $T_1,\ldots,T_{n-1}$ satisfying the following relations:
\begin{align}
T_iT_j = T_jT_i,\ & \mathrm{if}\ |i-j| > 1,\nonumber \\
T_iT_{i+1}T_i = T_{i+1} T_i T_{i+1},\ & \mathrm{for}\ 1 \leqslant i \leqslant n - 2, \\
(T_i+1)(T_i-q) = 0,\ & \mathrm{for}\ 1\leqslant i \leqslant n-1. \nonumber 
\end{align}
\end{defn}

\noindent The $q$-analogue of the random-to-top shuffle in $\HH_n(q)$, studied independently in \cite{Lusztig03} and \cite{qr2r}, is obtained from the right action of the element
\begin{equation}
\label{eq:TdefHecke}
\mathcal{T}(q) := \frac{1}{[n]_q} \sum_{i=1}^{n} T_{i-1}\cdots T_1,
\end{equation}
where
\begin{equation}
\label{equation.q integer}
	[n]_q = \frac{q^n-1}{q-1} =1 + q + \cdots + q^{n-1}
\end{equation}
is the $q$-analog of the integer $n$. The \emph{q-factorial} is $[m]_q! = [1]_q [2]_q \cdots [m]_q$.

Just as the Tsetlin library introduces rates as in~\eqref{equation.Tx} to enrich the random-to-top process, 
in this paper, we introduce additional rates $x_i$ to $\mathcal{T}(q)$ and an operator $ \mathcal{X}(q,\x)$, and write
\begin{equation}
\label{eq:TxdefHecke}
	\mathcal{T}(q,\x) := \sum_{i=1}^{n} \,  T_{i-1}\cdots T_1\, \mathcal{X}(q,\x).
\end{equation}
 Our three $q$-generalizations of the Tsetlin library are obtained from the action of $\mathcal{T}(q,\x)$ on 
 \begin{enumerate}
     \item permutations $\symm_n$ in \cref{section.qTsetlin on perms}, 
     \item words $W_\m$, i.e. permutations of sequences with repeated entries with multiplicities $\m=(m_1,\ldots,m_\ell)$ and $|\m|=n$ 
in \cref{section.qTsetlin on words}, and
\item  in \cref{section.qTsetlin on flags} on cosets $G/B$ where $G  = \GL_n(\F_q)$ and $B$ is the subgroup of upper triangular matrices in $G$, 
equivalent to an action on flags of the vector space $\F_q^n$.
 \end{enumerate}
We state our results on convergence to stationarity in \cref{section.mixing}.

In the context of permutations $\symm_n$, the relationship to the original Tsetlin library and random-to-top shuffle is depicted in \cref{figure.notation}. 
The operator $\mathcal{X}(q,\x)$ will be defined depending on these three settings. 

We show in~\cref{theorem.topcommutativediagram,theorem.bottomcommutativediagram} that with
 appropriate definitions of projectors $\projs$, inclusions $\incls$, and 
probabilities $\x$, the diagram in 
\cref{figure.commuting diagram} commutes. This allows us to prove that the eigenvalues and stationary distribution of the $q$-Tsetlin library on 
permutations and words can be computed using semigroup methods at the level of flags. We now summarize our main results.
  
\subsection{$q$-Tsetlin library on permutations}
\label{section.qTsetlin on perms}
We first describe our definition of the Tsetlin library acting on permutations, and corresponding results on the stationary distribution and characteristic polynomial. 

The action of the Hecke generator $T_{i}$ on permutations $\pi\in \symm_n$ is given by
\begin{equation}
\label{eq:heckeonperm}
	\pi \cdot T_i := \begin{cases}
	q \, \pi s_i \quad & \text{if }\pi_{i+1} < \pi_i, \\
	\pi s_i + (q-1) \, \pi & \text{if }\pi_{i+1} > \pi_i.
\end{cases}
\end{equation}
Hence the action of $q^{-1}T_i$ acquires a probabilistic meaning when $q$ is real and
\begin{equation}
\label{equation.q restriction}
0 < q^{-1} \leqslant 1.
\end{equation}

We write $\Tq$ for the $q$-generalization of the random-to-top shuffle \eqref{eq:R2T}, i.e. when $\mathcal{T}(q)$ defined 
in \eqref{eq:TdefHecke} acts on permutations from the right. The operator $\Tq$ was studied from a representation theoretic perspective 
in~\cite{qr2r,BraunerComminsReiner}. 

\subsubsection{Definition of the $q$-Tsetlin library on permutations}
To obtain the $q$-Tsetlin library on permutations, we assign probabilities by defining $\X(q,\x)$ as
\begin{equation}
\label{defXq}
\pi \cdot \X(q,\x) :=  \frac{x_{\pi_1}}{q^{n-\pi_1}} \pi,
\end{equation}
where $x_1+\cdots + x_n=1$.
We then have $\Tqx$ defined as a right operator
\begin{equation}
\label{eq:Txdef}
\Tqx := \Tq \X(q,\x).
\end{equation} 
Our choice of convention to leave a power of $q$ explicit will become clear below, see \cref{remark.linemultiplicty}.
This clearly reduces to \eqref{eq:X} at $q=1$ and to \eqref{eq:TdefHecke} when setting $x_i = q^{n-i}/[n]_q$. 
We call the Markov chain obtained by the right action of $\Tqx$ on $\mathbb{C}(q,\x) \symm_n$ 
the \emph{$q$-Tsetlin library on permutations}.

\begin{example}
\label{ex:hecke n=3}
Let us consider the $q$-Tsetlin library $\Tqx$ for $n=3$. On the ordered basis of permutations 
$(123, 132, 213, 231, 312, 321)$
in $\mathbb{C}(q,x_1,x_2,x_3)\symm_3$, we can represent $\Tqx$ by the following matrix
\begin{equation}
\label{matrixn=3}
\Tqx = 
\begin{pmatrix}
 \frac{(q^2-q+1) x_1}{q^2} &  \frac{(q-1) x_1}{q^2} & x_2 & 0 & x_3 & 0 \\
 \frac{(q-1) x_1}{q}& \frac{x_1}{q} & x_2 & 0 & x_3 & 0 \\
x_1 & 0 & \frac{x_2}{q} &  \frac{(q-1)x_2}{q}  & 0 & x_3 \\
x_1 & 0 &  0& x_2& 0 & x_3 \\
0 & x_1 & 0 & x_2 & x_3 & 0 \\
0 & x_1 & 0 & x_2 & 0 & x_3
\end{pmatrix}\,.
\end{equation}
\end{example}

\subsubsection{Stationary distribution}
\label{section.stationary_perms}

Recall that a Markov chain is said to be \emph{irreducible} if there is a sequence of moves taking any configuration to any other configuration.
It is a standard fact~\cite{LevinPeresWilmer} that an irreducible Markov chain has a 
unique stationary distribution. Since the transitions of the $q$-Tsetlin library contain the transitions of the Tsetlin library, which is known to be irreducible, 
we immediately obtain that (for $q\geqslant 1$) the $q$-Tsetlin library is irreducible, and hence has a unique stationary distribution. From general theory again, the \emph{stationary distribution} of the $q$-Tsetlin library is the left eigenvector $\Psiqx$ satisfying
\begin{equation}
\label{equation.stationary}
	\Psiqx \cdot \Tqx = \left(\sum_{i=1}^n x_i\right) \Psiqx.
\end{equation}
To formulate the result for the stationary distribution in \cref{th:eigenvectorH}, we need the following two definitions.

\begin{defn}
Let $\bv = (b_1, \dots, b_k)$ be a weakly decreasing tuple of integers. Define the function $\kappa$ by setting
\begin{equation}
\label{eq:kappadef}
\kappa(\bv;\x) := \sum_{i=1}^{k} x_{b_i} q^{i + b_i-k-1}.
\end{equation}
In particular, $\kappa(b_1;\x) = x_{b_1} q^{b_1-1}$.
For the long permutation $w_0 = n\; n-1\; n-2\; \ldots\; 1$
(written in one-line notation) we have
\begin{equation}
\label{equation.kappa omega0}
\kappa(w_0;\x) = 
x_1 + x_2 + \cdots + x_n.
\end{equation}
Denote by $\bv_>$ the elements of the tuple $\bv$ ordered in weakly decreasing fashion. To move from a weakly decreasing tuple to an
arbitrary tuple of $k$ integers $\bv$, we define $\kappa(\bv;\x) := \kappa(\bv_>;\x)$.
\end{defn}
\begin{defn}
A \emph{left-to-right minimum} in a permutation $\pi\in \symm_n$ is an element $\pi_j$ such that $\pi_i > \pi_j$ for all $1 \leqslant i < j$. 
Let $\LRM(\pi)$ be the set of positions of left-to-right minima of $\pi$. 
\end{defn}
It is well-known that the number of permutations in $\symm_n$ with $k$ left-to-right minima is the \textit{unsigned Stirling number of the first kind}, $s(n,k)$.

\begin{example}
Suppose $\pi=5426314$. Then the left-to-right-minima are $5,4,2,1$ in positions $1,2,3,6$. Hence $\LRM(\pi)= \{1, 2, 3, 6\}$.
 \end{example}

Suppose $k \notin \LRM(\pi)$. 
Then there exist $i < k$ positions such that $\pi_i < \pi_k$. Write $p_k = p_k(\pi)$ to be the smallest such position.
 Finally, we define $\inv(\pi)$ to be the number of inversions of $\pi\in \symm_n$, that is, the number of pairs $1\leqslant i<j\leqslant n$ such that 
$\pi_i>\pi_j$.

\begin{theorem}
\label{th:eigenvectorH}
Let $q \geqslant 1$ be a real number, so that $q^{-1} \subset (0,1]$ is a probability. 
\begin{enumerate}
\item For $\pi \in \symm_n$, the component $\Psiqx_\pi$ of the stationary distribution of the $q$-Tsetlin library $\Tqx$ 
on permutations is given by
\begin{equation}
\label{eq:ssH}
\begin{split} 
\Psiqx_\pi =& 
\frac{ \displaystyle q^{-\inv(\pi)} }
{\displaystyle \prod_{k=1}^{n-1}  
\left(\kappa(w_0;\x)-q^{k-n-1} \kappa(\pi_1, \dots, \pi_{k-1};\x) \right)} \\
& \times
\left( \prod_{\substack{k=1 \\ k \in \LRM(\pi)}}^{n-1} \kappa(\pi_k;\x) \;
\prod_{\substack{k=1 \\ k \notin \LRM(\pi)}}^{n-1} 
(\kappa(\pi_{p_k}, \dots, \pi_k;\x) 
- q^{-1} \kappa(\pi_{p_k}, \dots, \pi_{k-1};\x) ) \right).
\end{split}
\end{equation}
\item Each factor of $\Psiqx_\pi$ is positive.
\item The entries of $\Psiqx$ sum to $1$.
\end{enumerate}
\end{theorem}

\cref{th:eigenvectorH} will be proved in \cref{ss:ss permutations} by lumping the stationary distribution
of the $q$-Tsetlin library on flags that will be introduced in \cref{section.qTsetlin on flags}. A subtlety is that the $q$-Tsetlin
library on flags is only defined for $q$ prime powers whereas the $q$-Tsetlin library on permutations is defined for arbitrary $q$
satisfying~\eqref{equation.q restriction}.

\begin{cor}
\label{cor:stationaryspecial}
    For the special case $x_i=q^{n-i}/[n]_q$ the stationary distribution is given by
    \[
    \Psi_\pi = \frac{q^{n(n-1)/2}}{[n]_q!} q^{-\inv(\pi)}.
    \]
\end{cor}

\subsubsection{Eigenvalues}

Our next theorem states that the eigenvalues of $\Tqx$ are natural $q$-deformations of the corresponding eigenvalues of $\Tx$ with matching multiplicities. 

\begin{theorem}
\label{thm:evalues Sn}\label{thm:evalues Sn2}

The eigenvalues of the $q$-Tsetlin library $\Tqx$ in~\eqref{eq:Txdef} are given as follows. 
For every subset $S \subseteq [n]$, written as $S = \{i_1 > i_2 > \cdots > i_k\}$, there is an eigenvalue
\[
\lambda_S(q,\x) = \sum_{j=1}^k \frac{x_{i_j}}{q^{n-i_j-j+1}}
\]
which occurs with multiplicity $d_{n-k}$ given in \eqref{derange}. Moreover, $\Tqx$ is diagonalizable whenever the $\x$ are generic, i.e. $\lambda_T(q,\x) = \lambda_S(q,\x)$ implies that $S = T$.
\end{theorem}

\cref{thm:evalues Sn} is proven in \cref{ss:eigenvalue perm}.

\begin{remark}
Note the following regarding \cref{thm:evalues Sn}:
\begin{enumerate}
\item When $k=0$, the eigenvalue $\lambda_\emptyset = 0$ occurs with multiplicity $d_n$.
\item When $k=n$, then  $\lambda_{\{1,2,\ldots,n\}} = x_1+\cdots+x_n$ occurs with multiplicity $d_0 = 1$.
\item
The derangement number $d_1$ is zero, and hence $\lambda_S$ does not occur as an eigenvalue if $|S| = n-1$.
\end{enumerate}
\end{remark}

\begin{example}
Continuing Example \ref{ex:hecke n=3}, the transition matrix for $\Tqx$ when $n=3$ has eigenvalues $0,0,\frac{x_1}{q^2},\frac{x_2}{q},x_3,x_1+x_2+x_3$ (stated with multiplicity).
The stationary distribution defined in~\eqref{equation.stationary}, written as a vector in the above basis, is given by
\begin{equation}
\label{eq:psiqexample}
\Psiqx= 
\tfrac{1}{x_1+x_2+x_3}
\left(
\renewcommand{\arraystretch}{1.5}
\begin{array}{c}
qx_1\frac{q(x_1+x_2) - x_1}{q^2(x_1+x_2+x_3) - x_1}\\
x_1\frac{(q-1)x_1 + q^2 x_3}{q^2(x_1+x_2+x_3) - x_1}\\
\frac{q x_1 x_2}{q(x_1+x_2+x_3)-x_2}\\
x_2 \frac{(q-1)x_2 + q x_3)}{q(x_1+x_2+x_3)-x_2}\\
\frac{x_1 x_3}{x_1+x_2}\\
\frac{x_2 x_3}{x_1+x_2}
\end{array}
\right)^T.
\end{equation}
\end{example}

\subsection{$q$-Tsetlin library on words}
\label{section.qTsetlin on words}

We now generalize the $q$-Tsetlin library from \cref{section.qTsetlin on perms} to the case where there are multiple copies of each book---or in other words, to act on words.
The $q=1$ Tsetlin library has been generalized to this situation; see \cite{ASST.2015} for two generalizations. 

We work with the 
alphabet $\{1,\ldots,\ell\}$. Let $\m = (m_1, \dots, m_\ell)$ be a tuple of positive integers summing to $n$. 
Let $W_{\m}$ be the set of all words in $\{1,\ldots,\ell\}$ of length $n$ with $m_j$ occurrences of the letter $j \in \{1,\ldots,\ell\}$. For example,
\[
W_{(2, 2)} = (1122, 1212, 1221, 2112, 2121, 2211),
\]
in lexicographic order.

\subsubsection{Definition of the $q$-Tsetlin library on words}
Formally, the Hecke generators $T_1, \dots, T_{n-1}$ act on words the same way as on permutations, given in \eqref{eq:heckeonperm}. 
More precisely, the generator $T_i$ for $1\leqslant i<n$ acts on $w=(w_1\ldots w_n) \in W_{\m}$ as
\begin{equation}
\label{eq:heckeonwords}
	w  \cdot T_i := \begin{cases}
	q (w_1\dots w_{i+1}\; w_i\dots w_n) \quad & \text{if } w_{i+1} \leqslant w_i, \\
	(w_1 \dots w_{i+1} w_i \dots w_n) + (q-1) w & \text{if } w_{i+1} > w_i.
\end{cases}
\end{equation}
Notice that we could have placed the equality sign in the second case as well without altering the meaning.
A probability $\xo_j$ is associated to every letter $j\in \{1,\ldots,\ell\}$ so that 
\[ \xo_1 + \cdots + \xo_\ell = 1.\]
We define the operator $\wordX(q,\xbar)$ to assign probabilities to a word $w$ as follows, with $[m]_q$ as in \eqref{equation.q integer}:
\[
	w \cdot \wordX(q,\xbar) := 
	\frac{\xo_{w_1}}{q^{m_{w_1+1}+\cdots+m_\ell}[m_{w_1}]_q} w.
\]

The $q$-Tsetlin library on words is then given by the right action of the element
\[
\wordTqx := \sum_{i=1}^{n} \,  T_{i-1}\cdots T_1\, \wordX(q,\xbar).
\]

\begin{example}
\label{ex:words122}
On the ordered basis of words 
$W_{(1,2)}=(122, 212, 221)$ we can represent $\mathcal{T}_{W_{(1,2)}}(q,\xbar)$ by the matrix
\begin{equation}
\label{matrix122}
\mathcal{T}_{W_{(1,2)}}(q,\xbar) = 
\begin{pmatrix}
 \xo_1 &  \xo_2 & 0\\
\xo_1 & \frac{\xo_2}{[2]_q} & q\frac{\xo_2}{[2]_q} \\
\xo_1 & 0 & \xo_2
\end{pmatrix}\,.
\end{equation}
\end{example}

\noindent When $q = 1$, the $q$-Tsetlin library is a special case of the promotion Markov chain on posets where the poset is a union of chains of lengths 
$m_1, \dots, m_\ell$~\cite[last paragraph of Section 5.1]{AyyerKleeSchilling.2014}. 
Therefore, the transitions here include the transitions of this 
promotion Markov chain. Arguing as in the beginning of \cref{section.stationary_perms}, for $q\geqslant 1$ the $q$-Tsetlin library on words is irreducible, 
and the stationary distribution is unique. 

\subsubsection{Stationary distribution}

We obtain the unique stationary distribution for $\wordTqx$ from the stationary distribution in \cref{th:eigenvectorH} for $\Tqx$.

The notation of \cref{section.stationary_perms} naturally generalizes to words instead of permutations.
Let $\m=(m_1,\ldots,m_\ell)$ be a composition and $\bv=(b_1,\ldots,b_k)$ ($k\leqslant n$) a subword of $W_\m$. Then if $\bv$ 
is weakly decreasing, $\kappa$ as defined in \eqref{eq:kappadef} generalizes in the context of words to 
\[
	\kappa(\bv;\xbar) := \sum_{i=1}^{k} \frac{q^{i+m_1+ \dots + m_{b_i}-k-1}}{[m_{b_i}]_{q}} \xo_{b_i}.
\]
If $b$ is not weakly decreasing, then set $\kappa(\bv;\xbar):=\kappa(\bv_>;\xbar)$ as before.

For $w \in W_{\m}$, we again say that $i$ is a \textit{left-to-right minimum} if $w_i$ is the weakly smallest among $(w_1, \dots, w_i)$.
We reuse the notation $\LRM(w)$ from \cref{section.stationary_perms}
for the set of left-to-right minima of $w$. For example,
$\LRM(2231) = \{1, 2, 4\}$. Suppose $k \notin \LRM(w)$. 
Then there exist positions $i < k$ such that $w_i < w_k$. Again, let 
$p_k = p_k(w)$ be the smallest such position.
The notion of inversion for words, also denoted $\inv$, generalizes that for permutations in the natural way. 
To be precise, we define $\inv(w) = \#\{(i,j) \in [n]^2 \mid i<j, w_i > w_j\}$.

\begin{theorem}
\label{th:steady state words}
Let $q \geqslant 1$ be a real number, so that $q^{-1} \subset (0,1]$. 
\begin{enumerate}
\item Let $\m = (m_1, \dots, m_\ell)$ be a tuple of positive integers. 
For $w \in W_{\m}$, the component $\Psi(q,\xbar)_w$ of the eigenvector $\Psi(q,\xbar)$ of the $q$-Tsetlin library $\wordTx$ satisfying
\begin{equation}
	\Psi(q,\xbar) \cdot \wordTx = \left( \sum_{i=1}^\ell \xo_i \right)  \Psi(q,\xbar),
\end{equation}
is given by
\begin{equation}
\begin{split}
\Psi(q,\xbar)_w =& 
\frac{ \displaystyle q^{-\inv(w)} \prod_{i=1}^\ell q^{-m_i+1} [m_i]_{q}!}
{\displaystyle \prod_{k=1}^{n-1}  
\left(\xo_1+\cdots+\xo_\ell-q^{k-n-1} \kappa(w_1, \dots, w_{k-1};\xbar) \right)} \\
& \times
\left( \prod_{\substack{k=1 \\ k \in \LRM(w)}}^{n-1} \kappa(w_k;\xbar) \;
\prod_{\substack{k=1 \\ k \notin \LRM(w)}}^{n-1} 
(\kappa(w_{p_k}, \dots, w_k;\xbar) 
- q^{-1} \kappa(w_{p_k}, \dots, w_{k-1};\xbar) ) \right).
\end{split}
\end{equation}
\item Each factor of $\Psi(q,\xbar)_w$ is positive.
\item The entries of $\Psi(q,\xbar)$ sum to $1$.
\end{enumerate}
\end{theorem}

\cref{th:steady state words} will be proved in \cref{ss:ss words} by lumping the $q$-Tsetlin library on permutations to the one on words.
The rates under this lumping are related via
\begin{equation}
\label{equation.rates xbar x}
	\xo_i = \frac{[m_i]_{q}}{q^{m_i-1}} x_{m_1+\cdots + m_{i-1}+1}.
\end{equation}

\begin{example}
\label{eg:evalues words}
For $\m = (1, 2)$, the stationary probabilities according to \cref{th:steady state words} are
\begin{equation}
\begin{split}
\Psi(q,\xbar)_{122} =& \frac{ \xo_1}
{\xo_1 + \xo_2}
= \frac{x_1}
{x_1+ (1+1/q) x_2},\\
\Psi(q,\xbar)_{212} =& \frac{(1 + 1/q) \bar{x}_1 \bar{x}_2 }
{(\xo_1 + \xo_2)((1+1/q)\xo_1 + \xo_2)}
= \frac{(1+1/q)x_1x_2 }
{(x_1 + x_2)(x_1+ (1+1/q) x_2)}, \\
\Psi(q,\xbar)_{221} =& \frac{\xo_2^2 }
{(\xo_1 + \xo_2)((1+1/q)\xo_1 + \xo_2) } = \frac{(1+1/q)x_2^2 }
{(x_1 + x_2)(x_1+ (1+1/q) x_2)},
\end{split}
\end{equation}
where (up to a power of $q$) the right hand sides arise from \cref{ex:hecke n=3} with $x_3= q^{-1} x_2$ and the identification
of rates as in~\eqref{equation.rates xbar x}. One may verify that these expressions indeed comprise the left stationary state 
of \eqref{matrix122} in \cref{ex:words122}.
\end{example}
    
\subsubsection{Eigenvalues}

To state the eigenvalues of the transition matrix for $\wordTqx$, we need some notation from the promotion Markov chain on linear 
extensions studied in \cite{AyyerKleeSchilling.2014}.
For our work here, it suffices to focus on posets which are
disjoint unions of chains. More precisely, let 
$P_{\m} = [m_1] + \cdots + [m_\ell]$
be the union of chains of
size $n$ whose elements are labeled consecutively within chains.
For a naturally labelled poset $P$, let $\mathcal{L}(P)$ be the set of linear extensions.
For example, $P_{(2, 2)}$ is the poset on $4$ elements labelled $[4]$ whose only covering relations are $1 < 2$ and $3 < 4$. 
Then,
\[
\mathcal{L}(P_{(2, 2)}) = 
\{1234, 1324, 1342, 3124, 3142, 3412\}.
\]

We say that $\pi \in \mathcal{L}(P)$ is a \emph{poset derangement} if it is a derangement when considered as a bijection from the set of labels of $P$ to itself. 
Let $d_P$ be the number of poset derangements of $P$.
For the above example, $d_{P_{(2, 2)}} = 2$ because
only two of the linear extensions above are derangements, 
namely $3142$ and $3412$.

Recall that an \emph{upper set} in a poset $P$ is a subset $S$ of $P$ such that if $x \in S$ and $x < y$ in $P$, then $y \in S$. If $S$ 
is an upper set in $P$, then $P \setminus S$ is also a poset in the obvious way.
When $P = P_{\m}$, all upper sets are determined by weak compositions $\av = (a_1, \dots, a_\ell)$, 
where $a_i \leqslant m_i$ is the number of elements in the $i$-th chain belonging to the upper set. \cref{fig:poset} gives an example 
of the poset $P_{(5,3,4,2)}$ and upper set $\av=(2,0,3,1)$.

\begin{theorem}
\label{thm:evalues words}
The eigenvalues of the transition matrix of the $q$-Tsetlin library $\wordT(q,\xbar)$ on $W_{\m}$ 
are as follows. For every upper set $S$ of $P_{\m}$ determined by the composition $\av = (a_1, \dots, a_\ell)$, 
there is an eigenvalue
\begin{equation}
\label{eq:simplifiedwordevalue}
\lambda_{\av}(q,\xo) = \sum_{j=1}^{\ell} \xo_{j} \frac{q^{(a_{j+1} + \cdots + a_\ell)}[a_j]_q}{q^{n-n_j}[m_\ell]_q}=
\xo_{\ell} \frac{[a_\ell]_q}{[m_\ell]_q} + \frac{\xo_{\ell-1}}{q^{m_\ell-a_\ell}} \frac{[a_{\ell-1}]_q}{[m_{\ell-1}]_q} 
+ \cdots + \frac{\xo_{1}}{q^{m_2-a_2 + \ldots + m_\ell-a_\ell}} \frac{[a_{1}]_q}{[m_{1}]_q}  
\end{equation}
with multiplicity $d_{P_{\m} \setminus S}$.
\end{theorem}

\begin{figure}
\begin{tikzpicture}
[
dot/.style = {circle, fill, minimum size=#1,
              inner sep=0pt, outer sep=0pt},
dot/.default = 5pt 
] 

\def\na{5};
\foreach \i in {1,...,\na} {
    \node [dot] (\i) at (0,\i) {};
    }
\draw (1) -- (\na);
\draw [thick, decorate, decoration = brace] (-0.3,1) --  (-0.3,\na) node[left, pos=0.5, left=0.1pt]{$m_1$};
\draw [thick, decorate, decoration = brace] (0.3,\na) --  (0.3,4) node[pos=0.5,right=2pt]{$a_1$};
 
\def\nb{3};
\foreach \i in {1,...,\nb} {
    \node [dot] (\the\numexpr\i+\na\relax) at (2,\i) {};
    }
\draw (\the\numexpr\na+1\relax) -- (\the\numexpr\na+\nb\relax);
\draw [thick, decorate, decoration = brace] (1.7,1) --  (1.7,\nb) node[left, pos=0.5, left=0.1pt]{$m_2$};

\def\nc{4};
\foreach \i in {1,...,\nc} {
    \node [dot] (\the\numexpr\i+\na+\nb\relax) at (4,\i) {};
    }
\draw (\the\numexpr\na+\nb+1\relax) -- (\the\numexpr\na+\nb+\nc\relax);
\draw [thick, decorate, decoration = brace] (3.7,1) --  (3.7,\nc) node[left, pos=0.5, left=0.1pt]{$m_3$};
\draw [thick, decorate, decoration = brace] (4.3,\nc) --  (4.3,2) node[pos=0.5,right=2pt]{$a_3$};

\def\nd{2};
\foreach \i in {1,...,\nd} {
    \node [dot] (\the\numexpr\i+\na+\nb+\nc\relax) at (6,\i) {};
    }
\draw (\the\numexpr\na+\nb+\nc+1\relax) -- (\the\numexpr\na+\nb+\nc+\nd\relax);
\draw [thick, decorate, decoration = brace] (5.7,1) --  (5.7,\nd) node[left, pos=0.5, left=0.1pt]{$m_4$};
\draw [thick, decorate, decoration = brace] (6.3,2.1) --  (6.3,1.9) node[pos=0.5,right=2pt]{$a_4$};

\end{tikzpicture}
\caption{The poset $P_{\m}$ with $\m=(5,3,4,2)$ and with upper set $\av=(2,0,3,1)$. Nodes are labeled from bottom to top.
\label{fig:poset}
}
\end{figure}

\cref{thm:evalues words} is proved in \cref{ss:eigenvalue word}.

\begin{example}
Consider the case of $\m = (3, 3)$, so that $\ell = 2$ and $n = 6$.
The eigenvalues and multiplicities are given in \cref{tab:evalues words}. 
The last row corresponding to $\av = (3,3)$ is the unique largest eigenvalue.
\begin{table}[h!]
    \centering
    \begin{tabular}{|c|c|c|}
    \hline
     Composition $\av$ & Eigenvalue    & Multiplicity \\
     \hline
    (0,0) & $0$    &  6 \\
    (1,0) & $\xo_1 /(q^{3}(1 + q +q^2))$     & 3 \\
    (0,1) & $\xo_2 /(1 + q +q^2)$    & 3 \\
    (2,0) &$\xo_1 (1+q) /(q^{3}(1 + q +q^2))$     &  1\\
    (0,2)&$\xo_2 (1+q)/(1 + q +q^2)$     &  1\\
    (1,1)&$(\xo_1 + \xo_2 q^{2})/(q^{2}(1 + q +q^2))$     & 2 \\
    (2,1)&$(\xo_1 (1 + q) + \xo_2 q^{2})/(q^{2}(1 + q +q^2))$    & 1 \\
    (1,2)&$(\xo_1 + \xo_2 (q + q^{2}))/(q (1 + q +q^2))$     & 1 \\
    (2,2)&$(1+q)(\xo_1 + \xo_2 q)/(q (1 + q +q^2))$    &  1\\
    (3,3)&$\xo_1 + \xo_2 =1$    &  1 \\
    \hline
    \end{tabular}
    \caption{The eigenvalues and multiplicities for $\wordT(q,\xbar)$ for $\m = (3, 3)$.
    \label{tab:evalues words}}
\end{table}
\end{example}

\subsection{$q$-Tsetlin library on cosets and flags}
\label{section.qTsetlin on flags}

As discussed in the introduction, a crucial step in our analysis of the $q$-Tsetlin library on permutations and words is the lumping from 
a Markov chain on the space of flags. We conclude this section with a description of this Markov chain, and the results we prove and 
utilize to understand $\Tqx$ and $\wordTqx$.

We will first establish some basic facts about the space of flags over $\F_q^n$. Let $G = \GL_n(\mathbb{F}_q)$ and $B$ (resp. $\bB$) be 
the Borel subgroup of upper (resp. lower) triangular matrices in $G$. 
The set $G/B$ is the set of cosets of $B$, and $\C[G/B]$ is the vector space of $\C$-linear combinations of these cosets. 
Define the set
\begin{equation} \label{eq:cosetreps}
    \mathfrak{G}:= \{ g \, \in G \mid \textrm{ the rightmost (resp. topmost) non-zero entry of $g$ in every row (resp. column) is 1}   \}.
\end{equation}
Then $\mathfrak{G}$ is a set of distinguished coset representatives for $G/B$.
Furthermore, any $g' \not \in \mathfrak{G}$ can be row reduced to an element $g \in \mathfrak{G}$ via right multiplication by an element of $B$. 

There is a further decomposition of $G/B$ by \emph{double cosets} as follows. 
Letting $\bB$ be as above, we consider the double coset $\bB \pi B$. A good representative for this set is a permutation matrix $\pi$; 
the elements in $\bB \pi B$ 
are the matrices such that the rightmost nonzero entries in each row form the permutation matrix $\pi$ and this entry is also the topmost 
nonzero entry in its column. With $\pi$ a permutation, we will use the notation 
\[ 
	[\pi] := \bB \pi B. 
\]
Here, we abuse notation slightly by thinking of $\pi$ on the left-hand-side as a permutation in 
$\symm_n$, written in one-line notation, whereas we view the distinguished representative of $\bB \pi B$ as a permutation matrix in $G$. 
Note that the size of the set $[\pi]$ is 
$q^{\coinv(\pi)}$, where $\coinv(\pi)$ is the number of \textit{coinversion} pairs, i.e. $(i, j)$ with $1 \leqslant i < j \leqslant n$ such that $\pi_i < \pi_j$. 
(The presence of coinversion rather than inversion is because we take double cosets with \emph{lower} triangular matrices $\bB \pi B$ rather 
than upper triangular matrices $B \pi B$.)

\begin{example}\label{ex:doublecosets}
In $\GL_3(\mathbb{F}_q)$, the double cosets $[\pi] = \bB \pi B$ for $\pi \in \symm_3$ are given as follows:
\begin{equation}
\label{eq:example3q}
\begin{split} 
& [123] = \left\{ \left.
\begin{pmatrix} 1 & 0 & 0 \\ \alpha & 1 & 0 \\ \beta & \gamma & 1
\end{pmatrix} B\, \right| \alpha, \beta, \gamma \in \mathbb{F}_q
\right\}, \\
& [213] = \left\{ \left.
\begin{pmatrix} 0 & 1 & 0 \\ 1 & 0 & 0 \\ \alpha & \beta & 1
\end{pmatrix} B\, \right| \alpha, \beta \in \mathbb{F}_q
\right\}, 
\quad [132] = \left\{ \left.
\begin{pmatrix} 1 & 0 & 0 \\ \alpha & 0 & 1 \\ \beta & 1 & 0
\end{pmatrix} B\, \right| \alpha, \beta \in \mathbb{F}_q
\right\}, \\
& [312] = \left\{ \left.
\begin{pmatrix} 0 & 1 & 0 \\ 0 & \alpha & 1 \\ 1 & 0 & 0
\end{pmatrix} B\, \right| \alpha \in \mathbb{F}_q
\right\},
\quad\quad
[231] = \left\{ \left.
\begin{pmatrix} 0 & 0 & 1 \\ 1 & 0 & 0 \\ \alpha & 1 & 0
\end{pmatrix} B\, \right| \alpha \in \mathbb{F}_q
\right\}, \\
& [321] = \left\{
\begin{pmatrix} 0 & 0 & 1 \\ 0 & 1 & 0 \\ 1 & 0 & 0
\end{pmatrix} B
\right\}.
\end{split}
\end{equation}
\end{example}

The set $G/B$ can equivalently be identified with the set of complete flags in the vector space $\mathbb{F}_q^n$. A \textit{flag} is a 
sequence $(V_0 \subseteq V_1 \subseteq \cdots \subseteq V_n)$ of subspaces of $\mathbb{F}^n_q$ such 
that $\dim(V_i)=i$ for $0\leqslant i \leqslant n$. Let
\begin{equation}
\label{equation.flags}
\mathfrak{F} := \{(V_0\subseteq V_1 \subseteq \cdots \subseteq V_n) \mid  \dim(V_i) = i \}.
\end{equation}
One can view a coset $gB$ for $g \in \mathfrak{G}$ as a flag $(V_0 \subseteq V_1 \subseteq \cdots \subseteq V_n) = \F_q^n$ 
in $\mathfrak{F}$ by setting $V_i$ to be the span of the first $i$ columns of $g$. In this paper, we move interchangeably between viewing cosets in $G/B$ 
as coset representatives and flags, as the three perspectives have different advantages depending on the context.  

\begin{example}\label{ex:lines}
The number of lines in $\mathbb{F}^n_q$ is $[n]_q$ and the number of flags in $\mathbb{F}^n_q$ is $[n]_q!$. Thus, the space $\mathbb{F}^2_2$ has $[2]_2=3$ lines:
\[
\langle e_1 \rangle,\ \langle e_2 \rangle\ \text{and}\ \langle e_1+e_2 \rangle,
\]
and there are $[2]_2! =3$ flags in $\mathbb{F}^2_2$; the corresponding coset representatives are to the left:
\begin{align*}
&
\begin{array}{l}
\{0\} \subseteq \langle e_1 \rangle \subseteq \langle e_1,e_2 \rangle\\
\{0\} \subseteq \langle e_1+e_2 \rangle \subseteq \langle e_1,e_2\rangle
\end{array},\qquad
\begin{pmatrix} 1 & 0  \\ \alpha & 1 
\end{pmatrix}\; (\alpha \in \{0,1\}),\\
&\;\;\{0\} \subseteq \langle e_2 \rangle \subseteq \langle e_1,e_2 \rangle,\qquad\qquad\;\;
\begin{pmatrix} 0 & 1  \\ 1 & 0 
\end{pmatrix}.
\end{align*}
\end{example}

One can identify the distinguished representative of $[\pi]$ given by a permutation matrix with a flag $F_\pi$.
For example, when $n=3$, we write below $F_\pi$ (left) and $|[\pi]| = \coinv(\pi)$ (right) for $q=2$:
\[
\begin{aligned}
F_{321} &= (\{0\} \subseteq \langle e_3 \rangle \subseteq \langle e_3,e_2 \rangle \subseteq V), \qquad && 1\ \text{flag},\\
F_{312} &= (\{0\} \subseteq \langle e_3 \rangle \subseteq \langle e_3,e_1 \rangle \subseteq V), && 2\ \text{flags},\\
F_{231} &= (\{0\} \subseteq \langle e_2 \rangle \subseteq \langle e_2,e_3 \rangle \subseteq V), && 2\ \text{flags},\\
F_{213} &= (\{0\} \subseteq \langle e_2 \rangle \subseteq \langle e_2,e_1 \rangle \subseteq V), && 4\ \text{flags},\\
F_{132} &= (\{0\} \subseteq \langle e_1 \rangle \subseteq \langle e_1,e_3 \rangle \subseteq V), && 4\ \text{flags},\\
F_{123} &= (\{0\} \subseteq \langle e_1 \rangle \subseteq \langle e_1,e_2 \rangle \subseteq V), && 8\ \text{flags}.
\end{aligned}
\]

Iwahori~\cite{Iwahori.1964} (see also Halverson--Ram~\cite{HalversonRam}) defines an action of the Hecke algebra on $\C[G/B]$ which we review in \cref{section.flags}. Anticipating the results from that section, the Hecke action on $\C[G/B]$ is most easily expressed in terms of an action on flags
\begin{equation}\label{eq:defheckeonflagaction}
(V_0 \subseteq \cdots \subseteq V_n) \cdot T_i = \sum_{W\neq V_i} (V_0 \subseteq \cdots \subseteq V_{i-1} \subseteq 
W \subseteq V_{i+1} \subseteq \cdots \subseteq V_n),
\end{equation} 
where the sum is over all subspaces $V_{i-1} \subseteq W \subseteq V_{i+1}$ such that $\dim( W) = i$ and $W \neq V_i$. 

\subsubsection{Definition of the $q$-Tsetlin library on cosets}

Recall from \eqref{eq:TdefHecke} that 
\[ 
	\mathcal{T}(q) := \frac{1}{[n]_q} \sum_{i=1}^n T_{i-1} \cdots T_1 .
\] 
As in the case of permutations and words, we 
write $\hatTq$ to denote the generator of this Markov chain acting on flags. Using \eqref{eq:defheckeonflagaction}, it was shown in~\cite[Prop 4.10]{qr2r} 
that $\hatTq$ has the following concrete combinatorial description:
\begin{equation}
\label{eq:r2tflags}
 	(V_0 \subseteq \cdots \subseteq V_n) \cdot \hatTq = \frac{1}{[n]_q} \sum_{\text{Lines } L \subseteq \F_q^n} 
 	(V_0 \subseteq L \subseteq V_1 + L \subseteq \cdots \subseteq V_{n-1} + L \subseteq V_n)\ \widehat{}\ ,
\end{equation}
where $\,{}^{\widehat{}} \,$ means that the duplicate subspace on the right hand side (which by definition will exist) is removed. In other words, 
$\hatTq$ acts on a flag by 
adding all possible lines in $\F_q^n$ to the front of the flag.
\begin{example}\label{ex:r2tonflags}
    Consider the case $n=2$ and $q=2$, and fix the flag $F = \big( \langle e_1 \rangle \subseteq \langle e_1, e_2 \rangle\big)$, where we dropped
    $V_0=\{0\}$ for ease of notation. Then recalling the three lines in $\F_2^2$ from \cref{ex:lines}, we have
    \begin{align*}
        F \cdot \hatTq = \frac{1}{3} \bigg(  \big( \langle e_1 \rangle \subseteq \langle e_1, e_2 \rangle\big) &+  \big( \langle e_1 + e_2 \rangle 
        \subseteq \langle e_1, e_2 \rangle \big) +  \big( \langle e_2 \rangle \subseteq \langle e_1, e_2 \rangle\big) \bigg).
    \end{align*}
\end{example}

The action of $\hatTq$ can be rephrased in terms of cosets as inserting a vector in the first column of a matrix representative, and removing column $k$ from the matrix, 
where $k$ is the smallest number such that column $k$ lies in the span of the first $k-1$ columns.
A similar Markov chain has been considered by Knutson~\cite{knutson.2018}; see also \cite{ayyer_linusson.2019}.

\begin{example}\label{ex:addlinestocoset}
Let $q=3$. Then adding the line $\langle (1,1,0)\rangle$ to the flag 
$$
\{ 0\} \subseteq \langle(0,1,1)\rangle\ \subseteq\ \langle (0,1,1),(1,0,2)\rangle \subseteq \mathbb{F}_3^3
$$
corresponds at the level of cosets to
\begin{equation}\label{eq:addlinecoset}
\left(\begin{array}{c|ccc} 1 & 0 & 1 & 0 \\ 1 & 1 & 0 & 0 \\ 0 & 1 & 2 & 1
\end{array}\right)B \mapsto 
\left(\begin{array}{ccc} 1 & 0 & 0 \\ 1 & 1 & 0 \\ 0 & 1 & 1
\end{array}\right) B,
\end{equation}
where we removed $\langle (1,0,2)\rangle = \langle (1,1,0)\rangle + 2 \langle (0,1,1)\rangle$ in $\mathbb{F}^3_3$.
\end{example}

In analogy to the $q$-Tsetlin library on permutations and words, we would now like to add probabilities to $\hatTq$. 
Our weight function is defined as follows
\begin{equation}\label{eq:weightfunctionlines}
L \cdot \mathcal{X}_{G/B}(q,\x) := \frac{x_i}{q^{n-i}} \, L \, \, \, \ \text{if}\ \, \, \, L = \langle e_i + \sum_{k=i+1}^n c_k e_k\rangle.
\end{equation}
We remark that every line here can be uniquely written as in the second equality in \eqref{eq:weightfunctionlines}. We write $[L] \big( L \cdot \mathcal{X}_{G/B}(q,\x)  \big)$ 
for the coefficient $x_i/q^{n-i}$ of $L$. Equivalently, we define $\mathcal{X}_{G/B}(q,\x)$ on cosets
\begin{equation}\label{equation.XGB}
    gB \cdot \mathcal{X}_{G/B}(q,\x) = \frac{x_i}{q^{n-i}} \, gB, 
\end{equation} 
where $i$ is the row index of the leading 1 in the first column of $g$. One can check that this function is equivalent to \eqref{eq:weightfunctionlines}, 
i.e. if $\langle L \rangle $ is the span of the first column of $gB$, then 
\[ 
	[gB] \big( gB \cdot \mathcal{X}_{G/B}(q,\x) \big) = [L] \big(L \cdot \mathcal{X}_{G/B}(q,\x) \big). 
\]
We then define the $q$-Tsetlin library on flags by the right action
\begin{equation}
    \hatTqx:=  \hatTq \, \mathcal{X}_{G/B}(q,\x).
\end{equation}

\begin{remark}
\label{remark.linemultiplicty}
Note that there are $q^{n-i}$ lines $L$ of the form as in~\eqref{eq:weightfunctionlines}
since there are $q$ choices for each coefficient $c_k$ for $i<k\leqslant n$.
For instance, in \cref{ex:addlinestocoset}, the element in \eqref{eq:addlinecoset} has weight $x_1/q^2$. It is for this reason that we keep 
a power of $q$ explicit in \eqref{defXq}.
\end{remark}

As in the case of $\hatTq$, we show in \cref{prop:flagactions} that $\hatTqx$ has the following combinatorial description on flags:
\begin{equation}
\label{eq:r2tflagsweighted}
 	(V_0 \subseteq \cdots \subseteq V_n) \cdot \hatTqx 
	= \sum_{\text{Lines } L \subseteq \F_q^n} \big((V_0 \subseteq L \subseteq V_1 + L \subseteq \cdots \subseteq 
	V_{n-1} + L \subseteq V_n)\ \widehat{}\ \big) \cdot \mathcal{X}_{G/B}(q,\x),
\end{equation}
which reduces to \eqref{eq:r2tflags} when we set $x_i=q^{n-i}/[n]_q$.
\begin{example}
    Continuing \cref{ex:r2tonflags} with $F = \big( \langle e_1 \rangle \subseteq \langle e_1, e_2 \rangle\big)$, we have
    \begin{align*}
        F \cdot \hatTqx &=  \bigg(  \frac{x_1}{2^{2-1}}\big(  \langle e_1 \rangle \subseteq \langle e_1, e_2 \rangle\big) +  \frac{x_1}{2^{2-1}}\big( \langle e_1 + e_2 \rangle \subseteq \langle e_1, e_2 \rangle \big) +  \frac{x_2}{2^{2-2}}\big( \langle e_2 \rangle \subseteq \langle e_1, e_2 \rangle\big) \bigg)\\
        &= \bigg(  \frac{x_1}{2}\big(  \langle e_1 \rangle \subseteq \langle e_1, e_2 \rangle\big) +  \frac{x_1}{2}\big( \langle e_1 + e_2 \rangle \subseteq \langle e_1, e_2 \rangle \big) +  x_2 \big( \langle e_2 \rangle \subseteq \langle e_1, e_2 \rangle\big) \bigg).
    \end{align*}
\end{example}

In fact, both $\hatTq$ and $\hatTqx$ are equivalent to generators of Markov chains defined by Brown~\cite{Brown} in the context of semigroups; 
see \eqref{eq:LRBflags} and \eqref{eq:weightedLRBflags} for a formal definition. This connection to semigroups will be a key tool in our analysis.

\subsubsection{Stationary distribution}

As mentioned above and explained in \cref{section.flags}, the process $\hatTqx$ can be shown to be equivalent to one arising from a left regular band 
algebra~\cite{Brown,BraunerComminsReiner}.
Hence we can employ the methods developed 
in~\cite{ASST.2015,RhodesSchilling.2019} to compute its stationary distribution. The generators of the left regular band are the $[n]_q$ 
lines in $\mathbb{F}_q^n$. 

Write $\Psi(q,\x)_F$ for the component of the stationary distribution $\Psi(q,\x)$ corresponding to the flag $F$. 

\begin{theorem}
\label{theorem.sd flags}
We have $\Psi(q,\x)_F = \Psi(q,\x)_{F'}$ for all flags $F,F'\in [\pi]$. Fixing $F_\pi\in [\pi]$, we have that 
\[
	\Psi(q,\x)_{F_\pi} = \prod_{k=1}^n \frac{ f(\pi_1,\ldots, \pi_k)}
	{\left(x_1 + \cdots + x_n - \sum_{s\in \{\pi_1,\ldots,\pi_{k-1}\}} \frac{x_s}{q^{n-s-b_{k}(s)}}\right)},
\]
where
\[
	f(\pi_1,\ldots,\pi_k) = \sum_{\substack{s\in \{\pi_1,\ldots,\pi_k\}\\ s\leqslant \pi_k}} 
	\frac{x_s}{q^{n-s-b_k(s)}(x_1+ \cdots + x_n)} (q-1)^{\chi(s<\pi_k)},
\]
$b_k(s) = \# \{ t\in \{\pi_1,\ldots,\pi_{k-1}\} \mid t>s \}$, and $\chi(s<\pi_k)=1$ if $s<\pi_k$ and 0 otherwise.
Moreover, when $x_1+\cdots+x_n=1$,
\[ 
	\sum_{F \in \mathfrak{F}} \Psi(q,\x)_{F} = 1.
\]  
\end{theorem}

Unlike in the $q$-Tsetlin libraries on permutations and words, it is not
straightforward to prove that the $q$-Tsetlin library on cosets is irreducible.
In \cref{ss:flags sd}, we will first prove the existence part of \cref{theorem.sd flags}. We will then rely on the 
projection in \cref{theorem.topcommutativediagram} to complete the proof of uniqueness.

\subsubsection{Eigenvalues}

The random-to-top Markov chain generated by $\hatTq$ on flags has nice eigenvalues by the results of Brown~\cite{Brown}; see also \cite{BraunerComminsReiner}. Using results in  \cite{Brown, bidigare_hanlon_rockmore.1999}, we compute the eigenvalues of the transition matrix of $\hatTqx$.

Recall the \emph{$q$-derangement numbers}~\cite{wachs-1989} given by
\begin{equation}\label{eq:qderangements}
d_n(q) = [n]_q! \sum_{k=0}^n \frac{(-1)^k}{[k]_q!} q^{\binom{k}{2}},
\end{equation}
where $[n]_q=1+q+q^2+\cdots+q^{n-1}$ as in~\eqref{equation.q integer}. When $q = 1$, this reduces to the formula for the usual derangement 
numbers given in \eqref{derange}. 

\begin{theorem}
\label{cor.eigenvalues R2T}
The eigenvalues of the $q$-Tsetlin library Markov chain generated by $\hatTqx$ on flags are as follows. For every subset $S \subseteq [n]$, written as 
$S = \{i_1 > i_2 > \cdots > i_k\}$, there is an eigenvalue
\[
\lambda_S(q,\x) = \sum_{j=1}^k \frac{x_{i_j}}{q^{n-i_j-j+1}}
\]
with (possibly null) multiplicity given by $d_{n-k}(q) q^{(n-i_1) + (n-1-i_2) + \cdots + (n-k+1-i_k)}$ when $k \neq n$, and with multiplicity $1$ when $k = n$. Moreover, $\hatTqx$ is diagonalizable whenever $q$ is a power of a prime number.
\end{theorem}

The proof of \cref{cor.eigenvalues R2T} is given in \cref{ss:eigenvalue flags}.
 
\begin{example} 
\label{example.q=2 n=3 flag}
Let $q=2$, $n=3$ and $x_1+ x_2 + x_3=1$. There are $2^3-1=7$ lines and $[3]_2! = 21$ flags. 
\cref{theorem.sd flags} thus gives:
\[
\begin{split}
\Psi(q,\x)_{[321]} &= \frac{x_3 x_2 x_1}{(1-x_3)(1-x_2-x_3)} = \frac{x_2 x_3}{x_1+x_2},\\
\Psi(q,\x)_{[312]} &= \frac{x_3 \frac{x_1}{2}\ (x_2+\frac{x_1}{2})}{(1-x_3)(1-\frac{x_1}{2}-x_3)} = \frac{x_1 x_3}{2(x_1+x_2)},\\
\Psi(q,\x)_{[231]} &= \frac{\frac{x_2}{2}(\frac{x_2}{2}+x_3)}{x_1+\frac{x_2}{2}+x_3} = \frac{x_2 (x_2+2x_3)}{2(2x_1+x_2+2x_3)},\\
\Psi(q,\x)_{[213]} &= \frac{\frac{x_1}{2} \frac{x_2}{2}}{x_1+\frac{y_2}{2}+x_3} = \frac{x_1x_2}{2(2x_1+x_2+2x_3)},\\
\Psi(q,\x)_{[132]} &= \frac{\frac{x_1}{4} (\frac{x_1}{4}+x_3)}{3\frac{x_1}{4}+x_2+x_3} = \frac{x_1(x_1+4x_3)}{4(3x_1+4x_2+4x_3)},\\
\Psi(q,\x)_{[123]} &= \frac{\frac{x_1}{4} (\frac{x_1}{4}+\frac{x_2}{2})}{3\frac{x_1}{4}+x_2+x_3} = \frac{x_1(x_1+2x_2)}{4(3x_1+4x_2+4x_3)},
\end{split}
\]
which, up to $q^{|[\pi]|}$ is equal to the stationary distribution given in \eqref{eq:psiqexample} specialized to $q=2$. The eigenvalues and 
their multiplicities are given in the following table according to \cref{cor.eigenvalues R2T}:
\[
\begin{array}{|c|c|}
\hline
    \text{Eigenvalues} & \text{Multiplicities} \\
    \hline
    0 & d_3(2) = 6 \\[0.2cm]
\displaystyle    \frac{x_1}{4} & d_2(2) 2^{3-1} = 8 \\[0.2cm]
\displaystyle    \frac{x_2}{2} & d_2(2) 2^{3-2} = 4 \\[0.2cm]
    x_3 & d_2(2) 2^{3-3} = 2 \\[0.2cm]
    1 = x_1 + x_2 + x_3 & 1 \\
    \hline
\end{array}
\]
\end{example}

\subsection{Convergence to stationarity}
\label{section.mixing}

In this section, we bound the rate at which the $q$-Tsetlin library on permutations and flags converges to its stationary distribution. 
A standard spectral notion is the \emph{relaxation time}~\cite[Section~12.2]{LevinPeresWilmer} given by the reciprocal of the 
\textit{spectral gap}, i.e. $1 - \lambda_2$, where $\lambda_2$ is the second-largest eigenvalue.
From \cref{thm:evalues Sn}, we obtain the following result.

\begin{cor}
\label{cor:relax}
Let 
\[
\lambda^* = \max_{\substack{S \subseteq [n] \\ |S| = n - 2}}\left\{ \lambda_S \right\}.
\]
Then the relaxation time is $t_{\text{rel}} = (1 - \lambda^*)^{-1}$. Furthermore, $\lambda^*$ is simple for special choices of $x_i$. We have
\begin{equation}
\label{eq:qspectralgap}
\begin{array}{ll}
\displaystyle
\lambda^* = \frac{[n-2]_q}{[n]_q} \qquad & \displaystyle x_i=\frac{q^{n-i}}{[n]_q},\\[5mm]   
\displaystyle \lambda^* = \frac{n-2}{n} &  \displaystyle x_i=\frac{1}{n}.
\end{array}
\end{equation}
\end{cor}

The proof of~\cref{cor:relax} is given in~\cref{section.convergence}.

\begin{remark}
In the case of $x_i=q^{n-i}/[n]_q$ the relaxation time determined by \eqref{eq:qspectralgap} converges to a constant, 
\[
t_\text{rel} \to 1/(1-q^{-2})\quad \text{as}\quad n\to\infty.
\]
In contrast, for $x_i=1/n$ the relaxation time is equal to that of the standard $q=1$ Tsetlin library with equal rates, for which the eigenvalues are known \cite{Phatarfod.1991}, and diverges linearly with $n$,
\[
t_{\text{rel}} = \frac{n}{2} \to \infty\quad \text{as}\quad n\to\infty.
\]
\end{remark}

The \textit{mixing time} measures the time required by a Markov chain for the distance to stationarity to be small, which we now define.
The \textit{total variation distance} between two probability distributions $\mu$ and $\nu$ on the same space $\Omega$ is
\begin{equation}
\label{tv dist}
\|\mu - \nu \| = \frac{1}{2} \sum_{\omega \in \Omega} |\mu(\omega) - \nu(\omega)|.
\end{equation}
We want to measure the worst total variation distance starting from any initial permutation to the stationary distribution $\Psi(q, \x)$ of the $q$-Tsetlin 
library on permutations with transition matrix $\T$. Let
\begin{equation}
\label{eq:dtdef}
d(t) = \max_{\omega \in \Omega} \| \T^t(\omega, \cdot) - \Psi \|,
\end{equation}
where $\T^t(\omega, \cdot)$ is the distribution of the chain starting with state $\omega$.
Fix a positive real number $\epsilon < 1/2$ (for example, $\epsilon=\frac{1}{4}$). The \emph{(total variation) mixing time} is then
\begin{equation}
t_{\text{mix}} \equiv t_{\text{mix}}(\epsilon) 
= \min \{t \geqslant 1 \mid d(t) \leqslant \epsilon \}.
\end{equation}
It is well-known that the mixing time is related to the 
relaxation time by~\cite[Section~12.2]{LevinPeresWilmer}:
\begin{equation}
\label{eq:tmixbound}
\log \left( \frac{1}{2\epsilon} \right) (t_{\text{rel}} - 1) \leqslant
t_{\text{mix}} 
\leqslant \log \left( \frac{1}{\epsilon \min(\Psi(q, \x))} \right) t_{\text{rel}}.
\end{equation}
The total variation mixing time of the Tsetlin library for equal weights is $\Theta(n \log n)$, see~\cite[Theorem 1.1, Remark 1]{diaconisICM}.
(See also~\cite{Fill.1996,Nestoridi.2019} for further results on mixing times for the Tsetlin library and generalizations.)

We give an upper bound on the total variation mixing time for the $q$-Tsetlin library on flags and permutations,
which will be proved in~\cref{section.convergence}.

\begin{theorem}
\label{theorem.mixing}
Let $p>0$. The total variation mixing time of the $q$-Tsetlin library with weights $x_i=p^{n-i}/[n]_p$ on flags over $\mathbb{F}_q^n$, as well as on permutations, is 
\begin{enumerate}
\item $O(n)$ for $p,q>1$;
\item $O(p^n)$ for $p>q=1$;
\item $O(n \log n)$ for $q\geqslant p=1$;
\item $O(p^{-n})$ for $q\geqslant 1>p>0$.
\end{enumerate}
\end{theorem}

\begin{remark}
\cref{theorem.mixing} demonstrates that there are phase transitions for the behavior of the total variation mixing time, in particular since the
mixing time for $q=p=1$ is sharp by~\cite{diaconisICM}. The behavior of the mixing time with respect to $q$ and $p$ is illustrated in 
Figure~\ref{figure.mixing}. We conjecture that these bounds are sharp.
\end{remark}

\begin{figure}[h]
\begin{tikzpicture}[x=4cm, y=3cm,scale=0.7]

    \coordinate (Origin) at (1,0);
    \coordinate (MaxQ) at (3,0);
    \coordinate (MaxP) at (1,2.5);
    \coordinate (Split) at (1,1);
    \coordinate (RightEnd) at (3,1);

    \fill[pattern={north west lines}, pattern color=orange!60] 
        (1,1) rectangle (3,2.5);
    
    \fill[pattern={north west lines}, pattern color=red!60] 
        (1,0) rectangle (3,1);

    \draw[{Bar[]}-Stealth, thick, black!70!black] (1,0) -- (MaxQ) node[below right] {$q$};
    \draw[{Bar[]}-Stealth, thick, green!50!black] (1,1) -- (MaxP) node[left] {$p$};
    \draw[thick, red!60!black] (1,0) -- (1,1);

    \draw[very thick, blue!80!black] (1,1) -- (RightEnd) node[right] {$O(n \log n)$};

    \fill[blue!80!black] (1,1) circle (3pt);
    \node[left, blue!80!black] at (1,1) {$1$};

    \node[below] at (1,0) {$1$};
    \node[left] at (1,0) {$0$};

    \node[orange!80!black, font=\large] at (1.7, 1.7) {$O(n)$};
    \node[red!80!black, font=\large] at (2.0, 0.5) {$O(p^{-n})$};
    \node[green!40!black, anchor=west, font=\large] at (0.55, 2.0) {$O(p^n)$};

\end{tikzpicture}
\caption{The behavior of the mixing time of~\cref{theorem.mixing} with respect to $q$ and $p$.
\label{figure.mixing}}
\end{figure}

\section{Relationship between the various $q$-Tsetlin libraries}
\label{section.flags}
Having defined the three processes $\Tqx$, $\wordTqx$ and $\hatTqx$, that act on permutations, words and flags, respectively, 
we now wish to understand their relationship to each other. 

In  \cref{subsection.commutativediagram}, we revisit the diagram in \cref{figure.commuting diagram} and 
\begin{enumerate}
\item rigorously define the projection and inclusion maps, and 
\item prove that the diagram commutes.
\end{enumerate}
This will allow us to precisely understand the relationship between the operators $\hatTqx$, $\Tqx$ and $\wordTqx$ defined 
in \cref{section.results}. Consequences of having such a commutative diagram, including various lumping statements, 
are made in \cref{subsection.diagramconsequences}. Finally, in \cref{ss:Brown LRB}, we explain the connection between $\hatTqx$ and semigroup theory.

\subsection{Definitions of the inclusion and projection maps}
\label{subsection.commutativediagram}
We will analyze the two layers of the commutative diagram in \cref{figure.commuting diagram} separately. 

\subsubsection{Commutative diagram between flags and permutations}

Recall from \cref{section.qTsetlin on flags} that the set $G/B$ can be decomposed into double cosets $[\pi]:=\bB \pi B$ for $\pi \in \symm_n$. 
These double cosets are at the core of the following maps. 

\begin{defn}
    Define the \textit{projection}
    \begin{align}
        \proj_{\symm_n}\colon  \C(q,\x)[G/B] &\longrightarrow \C(q,\x)[\symm_n] 
        &&\text{by mapping}\\
        gB &\longmapsto \pi \nonumber 
        &&\text{for } gB \in [\pi]
    \end{align}
    and \textit{inclusion}
    \begin{align}
        \incl_{G/B}\colon \C(q,\x)[\symm_n] &\longrightarrow \C(q,\x)[G/B]
        \qquad \text{by mapping}\\
        \pi &\longmapsto q^{\inv(\pi)} \sum_{gB \in [\pi]} gB. \nonumber
    \end{align}
\end{defn}
\begin{remark}
    Note that in this context, $q$ is a fixed power of a prime number, and thus $\C(q) = \C$. However, we write $\C(q,\x)$ rather than $\C(\x)$ for consistency. 
    Also note that $\proj_{\symm_n}(q) = q = \incl_{G/B}(q)$.
\end{remark}
\begin{example}
Consider $\pi = 213$, and suppose $q = 2$. Then, using \cref{ex:doublecosets}, we have that 
\[ [213] = \left\{ \left.
\begin{pmatrix} 0 & 1 & 0 \\ 1 & 0 & 0 \\ \alpha & \beta & 1
\end{pmatrix} B\, \right| \alpha, \beta \in \mathbb{F}_q
\right\}. \]
Therefore at $q=2$,
\[ \incl_{G/B}(213) = q \cdot \left(  \begin{pmatrix} 0 & 1 & 0 \\ 1 & 0 & 0 \\ 0 & 0 & 1
\end{pmatrix} B + \begin{pmatrix} 0 & 1 & 0 \\ 1 & 0 & 0 \\ 1 & 0 & 1
\end{pmatrix} B + \begin{pmatrix} 0 & 1 & 0 \\ 1 & 0 & 0 \\ 0 & 1 & 1
\end{pmatrix} B + \begin{pmatrix} 0 & 1 & 0 \\ 1 & 0 & 0 \\ 1 & 1 & 1
\end{pmatrix} B \right).\]
We stress that the above cannot be further simplified as it is an expression in $\C[G/B]$ and so addition and multiplication is in relation to cosets.
\end{example}

\begin{theorem}
\label{theorem.topcommutativediagram}
With $\proj_{\symm_n}$ and $\incl_{G/B}$ as above, the following diagrams commute:
\[
\begin{tikzcd}[column sep=huge, row sep=huge]
\C(q,\x)[G/B] \arrow[d, swap, "\proj_{\symm_n}"]  \arrow[r,  "\hatTqx" ]  
& \C(q,\x)[G/B]  \arrow[d, "\proj_{\symm_n}"] \\
\C(q,\x)[\symm_n]   \arrow[r,  "\Tqx" ] 
&  \C(q,\x)[\symm_n] 
\end{tikzcd}
\qquad \qquad 
\begin{tikzcd}[column sep=huge, row sep=huge]
\C(q,\x)[G/B] \arrow[r,  "\hatTqx"] 
& \C(q,\x)[G/B] \\
\arrow[u, hook, "\incl_{G/B}"] \C(q,\x)[\symm_n] \arrow[r,  "\Tqx" ]
&  \C(q,\x)[\symm_n]. \arrow[u, hook, swap, "\incl_{G/B}"] 
\end{tikzcd} 
 \]
\end{theorem}

For ease of notation, we let $\proj = \proj_{\symm_n}$ and $\incl = \incl_{G/B}$ when it is clear from context which maps we are referring to.
We will prove \cref{theorem.topcommutativediagram} by proving that the inclusion and projection maps commute with both the action of 
$\HH_n(q)$ and the weight map $\mathcal{X}$.

We begin with the action of $\HH_n(q)$ on the space $\C[G/B],$ following work of Iwahori~\cite{Iwahori.1964} (see also Halverson--Ram~\cite{HalversonRam}). For each $1 \leqslant  i \leqslant n - 1$ and each $t \in \mathbb{F}_q$, let
\begin{equation}
\label{def hi}
f_i(t) := \begin{pmatrix}
 1 & \\
 & \ddots \\
 & & 1 & 0 \\
 & & t & 1 \\
 & & & & \ddots \\
 & & & & & 1
\end{pmatrix}
\end{equation}
be the matrix with $1$'s on the diagonal, $t$ in the $(i + 1, i)$ entry and zeros everywhere else. 
Using this, Iwahori~\cite{Iwahori.1964} proved that the right action of a generator $T_i$ in $\HH_n(q)$ is given by
\begin{equation}
\label{Ti action cosets}
(g B)T_i =gs_iB + \sum_{t\neq 0} g f_i(t) B.
\end{equation}

We would like to understand how \eqref{Ti action cosets} interacts with the partition of $G/B$ by double cosets. \cref{lemma:actionofTiondoublecosets} is likely implicit in the literature on $\HH_n(q)$, but we include it here for completeness. 
\begin{lemma}\label{lemma:actionofTiondoublecosets}
Let $q$ be a prime power, $\sigma \in \symm_n$ and suppose $hB \in [\sigma]$. Then for $1 \leqslant i \leqslant n-1$, we have
\begin{equation}
\label{eq:tioncosetrep}
(hB) \cdot T_i = \begin{cases}
\displaystyle   \sum_{j = 1}^q g_j B \textrm{ where $g_jB \in [\sigma s_i]$ and distinct for $j\in\{1,\ldots,q\}$} 
& \text{if $\sigma_{i+1} < \sigma_i$,} \\
&\\
\displaystyle   gB + \sum_{j = 1}^{q-1} f_j(t)B \textrm{ for some $gB \in [\sigma s_i]$ and distinct $f_j(t)B \in [\sigma]$} 
& \text{if $\sigma_{i+1} > \sigma_i$.}
    \end{cases}
\end{equation}
Moreover, 
\begin{equation} \label{eq:tionsumofcosets}
\left( \sum_{hB \in [\sigma]} hB \right) \cdot T_i = \begin{cases}
\displaystyle  \sum_{gB \in [\sigma s_i]} gB & \text{if $\sigma_{i+1} < \sigma_i$,} \\
&\\
\displaystyle  q \left(\sum_{gB \in [\sigma s_i]} gB \right)+ (q-1) \left( \sum_{hB \in [\sigma ]} hB\right) 
& \text{if $\sigma_{i+1}> \sigma_i$.}
\end{cases}
\end{equation}
\end{lemma}

\begin{proof}
First, suppose $\sigma_{i+1} < \sigma_i$. Then $h$ in \eqref{eq:tioncosetrep} necessarily has the form
\[
h = \begin{blockarray}{ccccc}
& & i & i+1 & \\
\begin{block}{c(cccc)}
& & \vdots & \vdots & \\
\sigma_{i+1} & \hdots & 0 & 1 & \\
& & \vdots & \vdots & \\
\sigma_i & \hdots & 1 & 0 & \\
& & \vdots & \vdots & \\
\end{block}
\end{blockarray}\;.
\]
The action of $T_i$ is given in \eqref{Ti action cosets}.
The first term is $h s_i B$, which is obtained by interchanging the columns $i$ and $i+1$ of $h$.
This is already an element in $[\sigma s_i]$.
Now consider the summand in the second term, namely 
$\sum_{j=1}^{q-1} h f_j(t) B$, where $f_i(t)$ is defined in \eqref{def hi}. Then,
\[
h f_i(t) = \begin{blockarray}{ccccc}
& & i & i+1 & \\
\begin{block}{c(cccc)}
& & \vdots & \vdots & \\
\sigma_{i+1} & \hdots & t & 1 & \\
& & \vdots & \vdots & \\
\sigma_i & \hdots & 1 & 0 & \\
& & \vdots & \vdots & \\
\end{block}
\end{blockarray}\;,
\]
which is not of the form of the distinguished coset representatives $\mathfrak{G}$ in \eqref{eq:cosetreps}. Because $t\neq 0$, multiplying by a suitable choice of upper triangular matrix in $B$ shows that 
\[
h f_i(t) B = \begin{blockarray}{ccccc}
& & i & i+1 & \\
\begin{block}{c(cccc)}
& & \vdots & \vdots & \\
\sigma_{i+1} & \hdots & 1 & 0 & \\
& & \vdots & \vdots & \\
\sigma_i & \hdots & \alpha & 1 & \\
& & \vdots & \vdots & \\
\end{block}
\end{blockarray} \; B,
\]
where $\alpha \neq 0$. Distinct values of $t$ necessarily 
lead to distinct values of $\alpha$, proving the first case of
\eqref{eq:tioncosetrep}. 
Thus $(h B)\cdot T_i$ is a sum of $q$ terms, where the only difference in the terms is the value $\alpha$ at position $(\sigma_i, i)$. 
Therefore, when we sum over all $hB \in [ \sigma ]$, we obtain all possible matrices $g B$, where $gB \in [\sigma s_i]$, proving the first case 
of \eqref{eq:tionsumofcosets}.

Now, suppose $\sigma_{i+1} > \sigma_i$. Then $h$ has the form
\[
h = \begin{blockarray}{ccccc}
& & i & i+1 & \\
\begin{block}{c(cccc)}
& & \vdots & \vdots & \\
\sigma_{i+1} & \hdots & 1 & 0 & \\
& & \vdots & \vdots & \\
\sigma_i & \hdots & * & 1 & \\
& & \vdots & \vdots & \\
\end{block}
\end{blockarray}\; ,
\]
where the element $*$ may or may not be zero. 
Again, $h s_i$ interchanges the columns $i$ and $i+1$, but this time $h s_i$ is not in $\mathfrak{G}$. However, there is a suitable choice of $b \in B$ such that $h s_i b = g$, where $gB \in [\sigma s_i]$ looks like
\begin{equation}
\label{eq:hsi}
g = \begin{blockarray}{ccccc}
& & i & i+1 & \\
\begin{block}{c(cccc)}
& & \vdots & \vdots & \\
\sigma_{i+1} & \hdots & 0 & 1 & \\
& & \vdots & \vdots & \\
\sigma_i & \hdots & 1 & 0 & \\
& & \vdots & \vdots & \\
\end{block}
\end{blockarray}\;.
\end{equation}
Finally,
\begin{equation}
\label{eq:hhi}
h f_i(t) B = \begin{blockarray}{ccccc}
& & i & i+1 & \\
\begin{block}{c(cccc)}
& & \vdots & \vdots & \\
\sigma_{i+1} & \hdots & 1 & 0 & \\
& & \vdots & \vdots & \\
\sigma_i & \hdots & t + * & 1 & \\
& & \vdots & \vdots & \\
\end{block}
\end{blockarray}\; B ,
\end{equation}
and so the values of $h f_i(t)B$ for all $q-1$ nonzero values 
of $t$ are distinct. This proves the second case of~\eqref{eq:tioncosetrep}.
Now, as we sum over the action of $T_i$ on all $hB \in [ \sigma]$, 
we consider the two terms separately. The first is
\[
\sum_{h B \in [\sigma]} h s_i B =
\sum_{\substack{hB \in [\sigma ] \\ h_{\sigma_i, i} = 0}} \, \, \,
\sum_{\alpha = 0}^q (h + \alpha E_{\sigma_i, i} ) s_i B,
\]
where $E_{j,i}$ denotes the matrix with a $1$ in the $(j, i)$-th entry and 0 everywhere else.
Now, $(h + \alpha E_{\sigma_i, i} ) s_i B = g B$,  where $g$
is given in \eqref{eq:hsi}. Thus, the inner sum gives $q(gB)$, and
\[
\sum_{hB \in [ \sigma ]} h s_i B =
q \sum_{g B \in [\sigma s_i ]} g B.
\]
The second term is a double sum of $h f_i(t) B$ over $hB \in [ \sigma]$ and $t \neq 0$, where $h f_i(t) B$ is given in 
\eqref{eq:hhi}. By a similar argument to the one sketched above,
one can see that each matrix $hB$ for $hB \in [\sigma]$
will appear $q-1$ times,
completing the proof.
\end{proof}

We then easily conclude the following. 
\begin{cor}
\label{cor:ticommutestop}
For $h \in \GL_n(\F_q)$ and $\pi \in \symm_n$, the following identities hold for $1 \leqslant i \leqslant n-1$:
\begin{align}
  \label{eq:projcommuteswithaction}  \proj( hB \cdot T_i) &= \proj( hB) \cdot T_i,\\
   \label{eq:inclcommuteswithaction} \incl(\pi \cdot T_i) &= \incl(\pi) \cdot T_i.
\end{align}
\end{cor}

\begin{proof}
We will show that \eqref{eq:projcommuteswithaction} follows from \eqref{eq:tioncosetrep} and \eqref{eq:inclcommuteswithaction} follows 
from \eqref{eq:tionsumofcosets}. Let $hB \in [\sigma]$.
First suppose $\sigma_{i+1} < \sigma_i$. Applying \eqref{eq:tioncosetrep} and the definition of the action of $T_i$ on permutations in \eqref{eq:heckeonperm} gives
\begin{align*}
    \proj(hB \cdot T_i) &= \proj\left(\sum_{\substack{j =1 \\ g_j B \in [ \sigma s_i ]}}^q g_j B\right) \\
    &= q \ \sigma s_i = \sigma \cdot T_i = \proj(hB) \cdot T_i.
\end{align*}

Similarly, if $\sigma_{i+1} > \sigma_i$, we have 
\begin{align*}
    \proj(hB \cdot T_i) &= \proj \left(   gB + \sum_{\substack{j =1 \\ h_j B \in [\sigma]}}^{q-1}  h_jB\right) \\
    &= \sigma s_i + (q-1) \sigma = \sigma \cdot T_i = \proj(hB) \cdot T_i.
\end{align*}

We now prove \eqref{eq:inclcommuteswithaction}. Again, first consider $\sigma\in \symm_n$ with $\sigma_{i+1} < \sigma_i$. 
Then $\inv(\sigma s_i) = \inv(\sigma) - 1$ since $\sigma$ has a descent at position $i$. Thus using \eqref{eq:tionsumofcosets}, we have
\begin{align*}
    \incl(\sigma \cdot T_i) &= \incl(q \  \sigma s_i) 
    = q \left( q^{\inv(\sigma)-1} \sum_{g B \in [\sigma s_i] } g B \right) 
    = q^{\inv(\sigma)}  \sum_{g B \in  [\sigma s_i]} g B \\
    &= q^{\inv(\sigma)} \left( \sum_{hB \in [\sigma]} hB \right) \cdot T_i 
    = \incl(\sigma) \cdot T_i.
\end{align*}
If $\sigma_{i+1} > \sigma_i$, then $\inv(\sigma s_i) = \inv(\sigma) + 1$, and 
\begin{align*}
    \incl(\sigma \cdot T_i)&= \incl( \sigma s_i + (q-1) \sigma)
    = \left(q^{\inv(\sigma)+1} \sum_{g B \in [\sigma s_i]} gB \right) + (q-1)  \left( q^{\inv(\sigma)} \sum_{h B \in [\sigma ]} hB \right)\\
    &= q^{\inv(\sigma)} \left( q \sum_{g B \in [\sigma s_i ]} gB + (q-1) \sum_{h B \in [\sigma ]} hB \right) 
    = q^{\inv(\sigma)} \left( \sum_{h B \in [\sigma]} hB \right) \cdot T_i 
    = \incl(\sigma) \cdot T_i,
\end{align*}
completing the proof.
\end{proof}

The last piece we need before we may conclude \cref{theorem.topcommutativediagram} is that the maps $\mathcal{X}_{\symm_n}$ 
and $\mathcal{X}_{G/B}$ also commute with the inclusion and projection maps.

\begin{lemma}\label{lemma.chiandtopproj}
For $gB \in G/B$ and $\sigma \in \symm_n$, the following equalities hold:
    \begin{align*}
        \proj_{\symm_n}(gB) \cdot \X(q,\x) &= \proj_{\symm_n}(gB \cdot \mathcal{X}_{G/B}(q,\x)),  \\
        \incl_{G/B}(\sigma) \cdot \mathcal{X}_{G/B}(q,\x) &= \incl_{G/B} (\sigma \cdot \X(q,\x)).
    \end{align*}
\end{lemma}

\begin{proof}
Both equalities follow by construction from the definitions of $\mathcal{X}_{\symm_n}(q,\x)$ and $\mathcal{X}_{G/B}(q,\x)$. 
Recall from \eqref{equation.XGB} that for $g \in G$, the weight operator $\mathcal{X}_{G/B}(q,\x)$ acts on $gB$ by 
$gB \cdot \mathcal{X}_{G/B}(q,\x) = q^{i-n} x_i \, gB$, where $i$ is the row index of the leading 1 in the first column of $g$. The claim follows 
by noting that under $\proj_{\symm_n}(gB)$, this entry is $\sigma_1$, where $\sigma$ is such that $gB \in [\sigma]$. 
Similarly, $\incl_{G/B}(\sigma)$ maps to the sum of elements in $G/B$ 
with row pivot $\sigma_1$ in the first column. Since $\X(q,\x)$ and $\mathcal{X}_{G/B}(q,\x)$ only introduce scalars and we have now observed that 
these scalars match, the claim follows.
\end{proof}

\begin{proof}[Proof of \cref{theorem.topcommutativediagram}]
The claim follows immediately by the definitions of $\hatTqx$ and $\Tqx$, applying \cref{cor:ticommutestop} and \cref{lemma.chiandtopproj}, 
and extending by linearity. In particular, we have shown for $gB \in G/B$ and $\pi \in \symm_n$ that 
\begin{align*}
    \proj_{\symm_n}(gB \cdot \hatTqx) &= \proj_{\symm_n}(gB) \cdot \Tqx, \\
    \incl_{G/B} (\pi \cdot \Tqx) &= \incl_{G/B}(\pi) \cdot \hatTqx,
\end{align*} 
thus establishing the result.
\end{proof}

\subsubsection{Commutative diagrams between permutations and words}

We now move to the relationship between $\Tqx$ and $\wordTqx$. To do so, we set some notation and definitions.

Recall that $\m = (m_1, \ldots, m_\ell)$ is a tuple of positive integers, where $m_1 + \cdots +m_\ell = n$ (so $\m$ is an
integer composition of $n$). 
Let $W_{\m}$ be the set of words of content $\m$. The Hecke action on $W_{\m}$ is 
given in~\eqref{eq:heckeonwords}, and extended linearly. Let 
\begin{equation}
 \label{equation.Midef}
 \begin{split}
    n_i:=&\, m_1 + m_2 + \cdots + m_i,\\
    M_i:=&\,  \{ n_{i-1} +1 , n_{i-1} +2, \ldots, n_i \}.
\end{split}
\end{equation} 
Note that $|M_i| = m_i$ and $n_\ell = n$. We consider the Young subgroup
\[ 
	\symm_\m:= \symm_{M_1} \times \symm_{M_2} \times \cdots \times \symm_{M_\ell},
\]
where $\symm_{M_i}$ is the symmetric group on the set $M_i$. Thus
\[ 
	\C[W_\m] \cong \C[\symm_n/\symm_\m].
\]
Suppose that $\sigma \in \symm_\m$, so that $\sigma$ can be written as a product 
\[ 
	\sigma = \sigma^{(1)} \sigma^{(2)} \cdots \sigma^{(\ell)} 
\]
for $\sigma^{(i)} \in \symm_{M_i}$. 
Then define the inversion number on $\symm_\m$ as follows:
\begin{equation}\label{eq:wordinv}
	\inv_{\m}(\sigma):= \sum_{i=1}^{\ell} \inv_i(\sigma^{(i)}),
\end{equation}
where $\inv_i(\sigma^{(i)})$ is the inversion number of $\sigma^{(i)}$ in $\symm_{M_i}$. Note that this is not the same as the inversion $\inv$ on words appearing in \cref{th:steady state words}, which one can think of as inversions on $\symm_n/\symm_\m$ rather than on $\symm_\m$.

We next define a method of turning a word into a permutation, and vice-versa.
\begin{defn}
Given $w \in W_{\m}$, let $\std(w)$ be the \emph{standardization of $w$} into a permutation by assigning, from left to right, 
the letters labeled $1$ in $w$ to the set $M_1$, the letters labeled $2$ to the set $M_2$, and so forth until the letters labeled 
$\ell$ are labeled by $M_\ell$. By construction, $\std(w) \in \symm_n$. 

Given $\sigma \in \symm_n$, define the $\m$-destandardization $\destd_\m$ of $\sigma$ by mapping the letters in $M_i$ to $i$, 
so that $\destd_\m(\sigma) \in W_\m$. 
\end{defn}

\begin{example}
Consider the word $w = 322211233$. Then
\[ \std(w) = 734512689.\]
On the other hand, consider $\pi = 56412378$ and $\m = (2,3,1,2)$. Then 
\[ \destd_{(2,3,1,2)}(\pi) = 23211244.\]
\end{example}

We use $\std$ and $\destd_\m$ to define maps between $\C[\symm_n]$ and $\C[W_\m]$.
In addition, to relate $\Tqx$ and $\wordTqx$, we require the weights $\x$ and $\xbar$ to be compatible. For ease of notation and exposition, we introduce the following change-of-variables:
\begin{align}
\label{eq:ydef}    y_i:=\;& \frac{x_i}{q^{n-i}} \smallskip \\ \smallskip
 \label{eq:yodef}   \yo_i:= \;& \frac{\xo_i}{q^{n-n_i} [m_i]_q}.
\end{align}

\begin{defn}
Let $\m = (m_1, \ldots, m_\ell)$ be a composition of $n$.
A system of weights $\x = (x_1, \ldots, x_n)$ is $\m$-compatible if 
\[ 
	 y_i = y_j  \qquad \text{whenever} \qquad i,j \in M_k \qquad \text{for} \qquad 1 \leqslant k \leqslant \ell. 
\]
In other words, the $\y$ coordinates must be equal on the sets $M_k$.

\end{defn}

With this, we can now define the inclusion and projection maps. To do so, we must specify how to map elements of $\symm_n$ 
and $W_\m$, as well as the scalars $\C(q,\x)$ and $\C(q,\xbar)$. 

\begin{defn}\label{def:projandinclforwords}
Given $\Tqx$ and a composition $\m = (m_1, \ldots, m_\ell)$ of $n$, suppose the weights $\x$ of $\Tqx$ are $\m$-compatible. 
Then we define the \textit{projection}
    \begin{align*}
        \proj_{W_\m}: \C(q,\x)[\symm_n] &\longrightarrow \C(q,\xbar)[W_\m]
        \qquad \text{by mapping}\\
        \sigma &\longmapsto \destd_\m(\sigma)\\
        x_{n_{j-1} + i} & \longmapsto \frac{q^{m_j-i}}{[m_j]_q} \xo_{j} \qquad \qquad \text{ for any} \qquad 1 \leqslant i \leqslant m_j
    \end{align*}
and the \textit{inclusion}
    \begin{align*}
        \incl_{\symm_n}: \C(q,\xbar)[W_\m] &\longrightarrow \C(q,\x)[\symm_n]
        \qquad \text{by mapping}\\
        w &\longmapsto \sum_{\tau \in \symm_\m} q^{-\inv_\m(\tau)} \big( \tau \cdot \std(w) \big) \\
      \xo_j &\longmapsto  [m_j]_q \, x_{n_j}.
    \end{align*}
\end{defn}

\begin{remark}
    Note that $\proj_{W_\m}$ and $\incl_{\symm_n}$ are algebra homomorphisms, where
    \[ \proj_{W_m}(p(q,\x) \cdot \sigma ) = \proj_{W_\m}(p(q,\x)) \cdot \proj_{W_\m}(\sigma)\]
    for $p(q,\x) \in \C(q,\x)$ and $\sigma \in \symm_n$ (and analogously for $\incl_{\symm_n}$).
\end{remark}

\begin{example}
Suppose $\m = (3,1)$ and $w = 1211$. Then 
\[ \incl_{\symm_n}(1211) = 1423 + q^{-1} 2413 + q^{-1}1432 +  + q^{-2}2431 + q^{-2}3412 + q^{-3}3421.  \]
\end{example}

We show in \cref{lemma:howtomapybar} below that the projection map (1) preserves the sum of the weights $\x$, and (2) is 
well-defined whenever $\x$ is $\m$-compatible. 

\begin{lemma}\label{lemma:howtomapybar}
Fix a composition $\m = (m_1, \ldots, m_\ell)$ of $n$, and suppose $\x = (x_1, \ldots, x_n)$ is a tuple that is $\m$-compatible. 
Then the following hold:
\begin{enumerate}
    \item\label{eq:xbarsumto1} Suppose $x_1 + x_2 + \cdots + x_n =1.$ 
    Then defining $\xbar = (\xo_1, \ldots, \xo_\ell)$ by 
    \[ \xo_j = [m_j]_q \ x_{n_j},\]
    we have that $\xo_1 + \xo_2 + \cdots +\xo_\ell = 1$.
  \item \label{eq:ybarmapping}
For $1 \leqslant i \leqslant m_j$ we have
\begin{align*}
     \proj_{W_\m}(y_{n_{j-1}+i}) &= \yo_j,\\
     \incl_{\symm_n}(\yo_j) &= y_{n_j}.
 \end{align*}
 Hence $\proj_{W_\m}$ is well-defined when $\x$ is $\m$-compatible.
\end{enumerate}
\end{lemma}

\begin{proof}
For \eqref{eq:xbarsumto1}, we have that $x_{i} = q^{n-i}y_i$. Since $\x$ is $\m$-compatible,  
\begin{align*}
    &y_{n_j} = \, \, y_{n_{j}-1} = y_{n_j-2} =\cdots = y_{n_{j-1}+1} \\
   =& \frac{x_{n_j}}{q^{n-n_j}} = \frac{x_{n_j-1}}{q^{n-n_j +1}} = \frac{x_{n_j-2}}{q^{n-n_j+2}} = \cdots = \frac{x_{n_{j-1}+1}}{q^{n-n_{j-1}-1}}, 
   \end{align*}
and thus $q^{i} x_{n_j} = x_{n_j-i}$ for $0 \leqslant i < m_j$.
It follows that 
\[
   x_{n_{j}}+ x_{n_{j}-1} + \cdots +  x_{n_{j-1}+1}  = x_{n_j} + q x_{n_j} + q^{2} x_{n_j} + \cdots +q^{m_j -1} x_{n_j} = [m_j]_q x_{n_j},
\]
and hence 
\[
   1= x_1 + \cdots + x_n = [m_1]_q \, x_{n_1} + [m_2]_q \, x_{n_2} + \cdots + [m_\ell]_q \, x_{n_\ell} 
    = \xo_1 + \cdots + \xo_\ell.
\]

For \eqref{eq:ybarmapping}, we have 
\begin{align*}
    \proj_{W_\m}(y_{n_{j-1}+i}) &= \frac{1}{q^{n-n_{j-1}-i}} \proj_{W_\m}(x_{n_j-1+i}) 
     = \frac{q^{m_j-i}}{q^{n-n_{j-1}-i}[m_j]_q} \xo_j \\
    & = \frac{q^{m_j -n + n_{j-1}}}{[m_j]_q} \xo_j 
    = \frac{\xo_j}{q^{n-n_j}[m_j]_q} =\yo_j.
\end{align*}

On the other hand, 
\[
    \incl_{\symm_n}(\yo_j) = \frac{1}{q^{n-n_j}[m_j]_q}\incl_{\symm_n}(\xo_j) 
    = \frac{1}{q^{n-n_j}} x_{n_j} = y_{n_j}.
\]

It follows that $\proj_{W_\m}$ is well-defined when $\x$ is $\m$-compatible, since for every $n_{j-1}+i \in M_{j}$, every $y_{n_{j-1}+i}$ is mapped 
to the same element. On the other hand, $\incl_{\symm_n}(\yo_j) = y_{n_j}$ determines the $\y$-coordinates because 
$y_{n_{j-1}+1} = y_{n_{j-1}+2} = \cdots = y_{n_j}$ by assumption.
\end{proof}

We are now ready to prove that the inclusion and projection maps commute with $\Tqx$ and $\wordTqx$.
\begin{theorem}\label{theorem.bottomcommutativediagram}
Suppose $\x$ is $\m$-compatible. Then
with $\proj_{W_\m}$ and $\incl_{\symm_n}$ as above, the following diagrams commute:
\[
\begin{tikzcd}[column sep=huge, row sep=huge]
\C(q,\x)[\symm_n] \arrow[d, swap, "\proj_{W_\m}"]  \arrow[r,  "\Tqx" ]  
& \C(q,\x)[\symm_n]  \arrow[d, "\proj_{W_\m}"] \\
\C(q,\xbar)[W_\m]   \arrow[r,  "\wordTqx" ] 
&  \C(q,\xbar)[W_\m] 
\end{tikzcd}
\qquad \qquad 
\begin{tikzcd}[column sep=huge, row sep=huge]
\C(q,\x)[\symm_n] \arrow[r,  "\Tqx"] 
& \C(q,\x)[\symm_n] \\
\arrow[u, hook, "\incl_{\symm_n}"] \C(q,\xbar)[W_\m] \arrow[r,  "\wordTqx" ]
&  \C(q,\xbar)[W_\m] \arrow[u, hook, swap, "\incl_{\symm_n}"] 
\end{tikzcd} 
 \]
\end{theorem}

\begin{proof}
We will again first show that the action of $T_i$ commutes with the maps $\proj_{W_\m}$ and $\incl_{\symm_n}$. We begin with 
$\proj_{W_\m}$. Note that $\destd_\m$ is order preserving in the sense that if $\sigma_{i+1} > \sigma_i$ and $w = \destd_\m(\sigma)$, 
then $w_{i+1} \geqslant w_i$ (and similarly if $\sigma_{i+1} < \sigma_i$). Thus it suffices to check the case when $w = \destd_\m(\sigma)$ 
has $w_{i+1} = w_i$. In this case, $\destd_\m(\sigma) \cdot T_i = q \ \destd_\m(\sigma)$. But we see that, using \eqref{eq:heckeonwords}, 
$\destd_\m(\sigma \cdot T_i) = q \ \destd_m(\sigma)$ as well, since by assumption $w_i = w_{i+1}$, implying that
     \[\destd_\m(\sigma \cdot s_i) = \destd_\m(\sigma) = \destd_\m(\sigma) \cdot s_i. \]

Next we consider the inclusion map $\incl_{\symm_n}$. Note that if $w_{i+1} < w_i$ in $w$ (respectively, $w_{i+1}> w_i$), and $\pi = \std(w)$, 
then $\pi_{i+1} < \pi_i$ (respectively, $\pi_{i+1} > \pi_i$), and this is preserved by the action of any $\sigma \in \symm_\m$, since $\sigma$ 
only permutes within blocks of elements that are equal in $w$. Hence, it again suffices to check the case that $w_{i} = w_{i+1}$; as before, 
$w \cdot T_i = q \ w$. 

Now consider $\incl_{\symm_n}(w) \cdot T_i$. Since $w_i = w_{i+1}$ by assumption, using the definition of $\incl_{\symm_n}(w)$ we may 
consider without loss of generality a word with a single alphabet, i.e. $\m = (n)$. In this case, $\std(w)$ is the identity permutation 
(in one-line notation).
On the one hand, we have 
\begin{equation*}
    \incl_{\symm_n}(w \cdot T_i) = q\ \incl_{\symm_n}(w) 
    = q\sum_{\sigma \in \symm_n} q^{-\inv(\sigma)} \sigma. 
\end{equation*}
On the other hand, let us consider $\incl_{\symm_n}(w) \cdot T_i$, which we break up into two sums:
\begin{align}
    \incl_{\symm_n}(w) \cdot T_i &= \left( \sum_{\sigma \in \symm_n} q^{-\inv(\sigma)}  \sigma \right) \cdot T_i \nonumber \\
    &= \label{eq:twosums}\left( \sum_{\substack{\sigma \in \symm_n \\ \sigma_{i+1} > \sigma_i}} 
    q^{-\inv(\sigma)}\sigma + \sum_{\substack{\nu \in \symm_n \\ \nu_{i+1}<\nu_i}} q^{-\inv(\nu)} \nu \right) T_i \\
    &= \label{eq:simpletwosum} \left( \sum_{\substack{\sigma \in \symm_n\\ \sigma_{i+1}> \sigma_i}} 
    (q^{-\inv(\sigma)} \sigma + q^{-\inv(\sigma s_i)} \sigma s_i )\right) \cdot T_i \\
    &= \label{eq:afterTiacts} \sum_{\substack{\sigma \in \symm_n\\ \sigma_{i+1}> \sigma_i}} 
    \left( q^{-\inv(\sigma)} \bigg(\sigma s_i + (q-1) \sigma \bigg) + q^{-\inv(\sigma s_i)} (q \sigma) \right), 
\end{align}
where \eqref{eq:simpletwosum} follows from the fact that there is a bijection between the two sums in \eqref{eq:twosums} from multiplication by 
$s_i$. Since $\sigma_{i+1}> \sigma_i$, it follows that $\inv(\sigma s_i) = \inv(\sigma) + 1$, and therefore $-\inv(\sigma) = -\inv(\sigma s_i) + 1$. 
We can thus rewrite \eqref{eq:afterTiacts} to obtain
 \begin{align*}
   \incl_{\symm_n}(w) \cdot T_i &= \sum_{\substack{\sigma \in \symm_n\\ \sigma_{i+1}> \sigma_i}} 
   (q^{-\inv(\sigma s_i) +1} \sigma s_i + q^{-\inv(\sigma)+1} \sigma - q^{-\inv(\sigma)} \sigma + q^{-\inv(\sigma )-1+1} \sigma) \\
   &=  \sum_{\substack{\sigma \in \symm_n\\ \sigma_{i+1}> \sigma_i}} 
   q (q^{-\inv(\sigma s_i)} \sigma s_i + q^{-\inv(\sigma)} \sigma) \\
   &= q \incl_{\symm_n}(w) = \incl_{\symm_n}(w \cdot T_i). 
 \end{align*}

Finally, we must show that for $\sigma \in \symm_n$ and $w \in W_\m$, the following equalities hold:
    \begin{align*}
        \proj_{W_\m}(\sigma)\cdot \wordX(q,\xbar) &= \proj_{W_\m}(\sigma\cdot \X(q,\x)),  \\
        \incl_{\symm_n}(w)\cdot \X(q,\x)&= \incl_{\symm_n} (w\cdot \wordX(q,\xbar)).
    \end{align*}
 As in the case of \cref{lemma.chiandtopproj}, this follows by construction from the definitions of $\wordX(q,\xbar)$ and 
 $\X(q,\x)$, as well as \cref{lemma:howtomapybar}. Specifically, $\incl_{\symm_n}(w)\cdot \X(q,\x)= \incl_{\symm_n} (w\cdot \wordX(q,\xbar))$
only when $\x$ is $\m$-compatible because 
 $\incl_{\symm_n}(\yo_j) = y_{n_j}$ determines the other $y_{n_{j-1}+i}$ for $1 \leqslant i \leqslant m_j$.
\end{proof}

\subsection{Consequences of the commutative diagrams}
\label{subsection.diagramconsequences}

We first address the consequences of the projection maps $\proj_{\symm_n}$ and $\proj_{W_\m}$, which will allow us to deduce the 
stationary distributions of $\Tqx$ and $\wordTqx$ from $\hatTqx$; this will be crucial in \cref{section.stationary distribution}.

Suppose $\mathcal{M}^A$ and $\mathcal{M}^B$ are irreducible Markov chains on state spaces $A$ and $B$, respectively.
Denote by $T^A$ and $T^B$ the corresponding transition matrices, and $T_{a_1,a_2}^{A}$ the $(a_1, a_2)$ entry of $T^A$ (and similarly for $T^B$). 
We say that the chain $\mathcal{M}^A$ \emph{lumps} or \emph{projects} 
onto the chain $\mathcal{M}^B$ if there exists a surjection $p \colon A \to B$ such that
\begin{equation}
\label{equation.lumping T}
	\sum_{a_2 \in p^{-1}(b_2)} T^A_{a_1,a_2} = T^B_{b_1, b_2}
\end{equation}
for all $b_1, b_2 \in B$ such that $a_1 \in p^{-1}(b_1)$. 
Let $\Psi^A$ and $\Psi^B$ be the respective stationary distributions.
It is a standard fact~\cite[Section 2.3]{LevinPeresWilmer} that in such a situation
\begin{equation}\label{eq:lumpingdef}
	\Psi^B_b = \sum_{a \in p^{-1}(b)} \Psi^A_a.
\end{equation}
In our context, the map $p$ is $\proj_{\symm_n}$ and $\proj_{W_\m}$, respectively. In fact, our analysis says something stronger than \eqref{eq:lumpingdef}.

\begin{lemma}\label{lemma:hinB-commutes}
Suppose $h\in \bB.$ Then for any $gB \in \C[G/B]$ we have 
\begin{equation}\label{eq:hinB-commutes}
     h \big( gB \cdot \hatTqx \big) = \big( (hg)B \big) \cdot \hatTqx.
\end{equation} 
\end{lemma}

\begin{proof}
Our proof of \eqref{eq:hinB-commutes} involves two observations:
\begin{enumerate}
    \item For \emph{any} $h \in \GL_n(\F_q)$, we have that 
    \[ h \big( gB \cdot \hatTq \big) = \big( hgB \big) \cdot \hatTq, \]
    since $\hatTq$ is an element of $\HH_n(q)$, and the left action of $\GL_n(\F_q)$ on $\HH_n(q)$ commutes with the right (multiplication) action of $\HH_n(q)$ by \eqref{eq:defheckeonflagaction}.
    \item Suppose $h \in B^-$.
    Then 
    \[ [gB] \big( gB \cdot \mathcal{X}_{G/B}(q,\x) \big) = [hgB] \big( hgB \cdot \mathcal{X}_{G/B}(q,\x) \big).   \]
    This follows from the fact that 
    $\proj_{\symm_n}(gB) = \proj_{\symm_n}(hgB)$ and the definition of $\mathcal{X}_{G/B}(q,\x)$ in \eqref{equation.XGB}.
\end{enumerate}
Note that $\hatTqx$ is the composition of $\hatTq $ followed by $\mathcal{X}_{G/B}(q,\x)$. Letting $c_{gB}$ be 
coefficients in $\C$, by (1) above, we have that 
\begin{align*}
    h \big((gB) \cdot \hatTq\big) &= h \left( \sum_{g'B \in G/B} c_{g'B} \, g'B \right) = \sum_{g'B \in G/B} c_{g'B} (hg') B \\
    &= \big( hg B\big) \cdot \hatTq = \sum_{g''\in G/B} c_{g''B} g''B
\end{align*}
and therefore $c_{g'B} = c_{g''B}$ whenever $hg' = g''$, or in other words, $c_{g'B} = c_{hg'B}$.
It follows that
\begin{align*}
    h \big( gB \cdot \hatTqx\big)&= h \big( gB \cdot (\hatTq \circ \mathcal{X}_{G/B}(q,\x) ) \big) 
    = h \left( \sum_{g'B \in G/B} c_{g'B} \, g'B \cdot \mathcal{X}_{G/B}(q,\x) \right) \\
    &= \sum_{g'B \in G/B} c_{hg'B} (hg'B) \cdot \mathcal{X}_{G/B}(q,\x) \\
    &= \big( (hg)B \big) \cdot \hatTqx,
\end{align*}
as desired.
\end{proof}

\begin{cor} 
\label{cor:lumping}
If $\proj_{\symm_n}(gB) = \proj_{\symm_n}(g'B)$, then 
\begin{equation} \label{eq:psiequalflags} 
	\Psi(q,\x)_{gB} = \Psi(q,\x)_{g'B}.
\end{equation}
Moreover, the Markov chain generated by $\hatTqx$ lumps to the one generated by 
$\Tqx$, and thus for $\pi \in \symm_n$,
    \begin{equation}
    \label{eq:flaglumping} 
    \Psi(q,\x)_{\pi} = \sum_{\substack{gB \in G/B\\ \proj_{\symm_n}(gB) = \pi}} \Psi(q,\x)_{gB} = q^{\coinv(\pi)} \Psi(q,\x)_{gB}.
    \end{equation}
 Analogously, if $\x$ is $\m$-compatible, then $\Tqx$ lumps to $\wordTqx$, where $\x$ and $\xbar$ are related as in \cref{def:projandinclforwords}.
 Thus for $w \in W_\m$
     \begin{equation}
    \label{eq:permlumping}
     \Psi(q,\xbar)_w = \sum_{\substack{\pi \in \symm_n\\ \proj_{W_\m}(\pi) = w}} \Psi(q,\x)_\pi.
     \end{equation} 
\end{cor}

\begin{proof}
We prove \eqref{eq:psiequalflags} using \cref{lemma:hinB-commutes}. If $\proj_{\symm_n}(gB) = \proj_{\symm_n}(g'B)$, then there is some 
$h\in \bB$ such that $hg = g'$. By definition of the stationary state,
\begin{align*}
    \Psi(q,\x) \cdot \hatTqx &= \Psi(q,\x). 
\end{align*}
Applying \cref{lemma:hinB-commutes} and extending by linearity, we have
\[  \bigg( h \Psi(q,\x) \bigg) \cdot \hatTqx =  h \bigg(  \Psi(q,\x) \cdot \hatTqx \bigg) = h \Psi(q,\x).  \]
Since the stationary distribution is unique, this implies that $h \Psi(q,\x) = \Psi(q,\x)$, from which it follows that 
\[ \Psi(q,\x)_{gB} = \Psi(q,\x)_{hgB} = \Psi(q,\x)_{g'B}.\]

The first equality in \eqref{eq:flaglumping} follows from the definition of lumping, taking $p = \proj_{\symm_n}$. The second equality then follows 
from \eqref{eq:psiequalflags} by noting that the size of the set $[\pi]$ is $q^{\coinv(\pi)}$. Similarly, when $\x$ is $\m$ compatible 
(so $\proj_{W_\m}$ is well-defined) \eqref{eq:permlumping} follows from the definition of lumping, taking $p = \proj_{W_\m}$.
\end{proof}

The inclusion maps $\incl_{G/B}$ and $\incl_{\symm_n}$ allow us to deduce spectral properties of $\Tqx$ and $\wordTqx$ from $\hatTqx$. 
Let $\Lambda_{G/B}$ be the set of eigenvalues of $\hatTqx$, and similarly for $\Lambda_{\symm_n}$ and $\Lambda_{W_\m}$. (Note that this set 
says nothing about eigenvalue multiplicity). 

\begin{cor}\label{cor:diagonalizability}
Whenever $q$ is a prime power,
we have the following containment of sets:
\[ \Lambda_{W_\m} \subseteq \Lambda_{\symm_n} \subseteq \Lambda_{G/B}.\]
Moreover, if $\hatTqx$ is diagonalizable, then so are $\Tqx$ and $\wordTqx$.
\end{cor}

\begin{proof}
This claim follows from the fact that by 
\cref{theorem.topcommutativediagram,theorem.bottomcommutativediagram}, 
\begin{equation}
\label{eq:inclusioncommutes} 
\begin{split}
	\incl_{G/B}(\pi) \cdot \hatTqx &= \incl_{G/B} (\pi \cdot \Tqx),\\
	\incl_{\symm_n}(w) \cdot \Tqx &= \incl_{\symm_n}(w \cdot \wordTqx).
\end{split}
\end{equation}
We can view the images of both inclusion maps as subspaces of $\C(q,\x)[G/B]$ and $\C(q,\x)[\symm_n]$, respectively. Write these subspaces as 
$\incl_{G/B}(\C[\symm_n])$ and $\incl_{\symm_n}(\C[W_\m])$. Equation \eqref{eq:inclusioncommutes} thus implies that $\hatTqx$ and $\Tqx$ 
act stably on $\incl_{G/B}(\C[\symm_n])$ and $\incl_{\symm_n}(\C[W_\m])$, respectively. Hence, the eigenvalues and diagonalizability of $\Tqx$ 
and $\wordTqx$ can be deduced from that of $\hatTqx$ in the context that it is well-defined, i.e. when $q = p^m$ for a prime number $p$. 
\end{proof}

We will expand upon \cref{cor:diagonalizability} once we have computed $\Lambda_{G/B}$ in \cref{section.eigenvalues}. In fact, we 
will be able to conclude that $\Tqx$ and $\wordTqx$ are diagonalizable for arbitrary, generic $q$ with a bit more work.

\subsection{Connection to a Markov chain from semigroup theory}
\label{ss:Brown LRB}
Having described the relationships between the Markov chains generated by 
$\hatTqx$, $\Tqx$ and $\wordTqx$, our next goal is to link these processes to a 
Markov chain on $\C[G/B]$ arising from the theory of \emph{left regular bands}. 

A left regular band is a monoid $M$ such that for any elements $a,b \in M$, one has 
\begin{equation} \label{eq:LRBdef}
aba = ab.\end{equation}
Note that this implies that $a^2 = a$ for any $a \in M$. One obtains a \emph{left regular band algebra} ${\bf{k}}M$ by taking ${\bf k}$ linear 
combinations of the elements of $M$ for ${\bf k}$ a commutative ring. 

In seminal work of Bidigare--Hanlon--Rockmore~\cite{bidigare_hanlon_rockmore.1999} 
and Brown~\cite{Brown}, the authors develop a combinatorial technique to compute the spectrum of any Markov chain arising as multiplication in a left 
regular band. This will be an important tool in our work, even though the Markov chains we study do not arise directly in this manner.  

The relevant left regular band in our context is called the \emph{$q$-free left regular band}, $\mathcal{F}_n(q)$, and we will take ${\bf k} = \C$. 
The elements of $\mathcal{F}_n(q)$ are \emph{partial flags} $(V_0 \subseteq V_1 \subseteq\cdots \subseteq V_i)$ where $V_j \subseteq \F_q^n$ has 
dimension $j$. Multiplication in $\mathcal{F}_n(q)$ is defined as follows:
\[ 
	(V_0 \subseteq V_1 \subseteq \cdots \subseteq V_i) \cdot (W_0 \subseteq W_1 \subseteq \cdots \subseteq W_k) 
	= (V_0 \subseteq V_1 \subseteq \cdots \subseteq V_i \subseteq V_i + W_1 \subseteq V_i + W_2 \subseteq \cdots \subseteq V_i+W_k)^{\widehat{}},  
\]
where $\, {}^{\widehat{}}\,$ denotes removing any repeated subspace in the flag. One can check that this satisfies~\eqref{eq:LRBdef}. 

\begin{example}
Suppose $q=2$ and $n=4$, and let 
\begin{align*}
    V &= \big(\langle e_1 + e_2 \rangle, \langle e_1 + e_2, e_2 \rangle\big),\\
    W &= \big( \langle e_2 \rangle, \langle e_3 \rangle \big).
\end{align*}
Then 
\begin{align*}
    V \cdot W &= \big(\langle e_1 + e_2 \rangle, \langle e_1 + e_2, e_2 \rangle, \langle e_3, e_1 + e_2, e_2 \rangle \big),\\
    W \cdot V &=  \big( \langle e_2 \rangle, \langle e_3 \rangle, \langle e_1 + e_2,e_2, e_3 \rangle \big).
\end{align*}
Hence we see this product is not commutative.
\end{example}

The Markov chain of interest is generated by an element of $\mathcal{F}_n(q)$, defined by Brown~\cite{Brown} as 
\begin{equation} \label{eq:Browntsetlin}
\sum_{\text{Lines } L \subseteq \F_q^n} x_L \ L,\end{equation}
where $x_L$ are coefficients in $[0, 1]$ summing to $1$. We are interested in two specializations of \eqref{eq:Browntsetlin}:
\begin{align}
    \tTq:=& \frac{1}{[n]_q} \sum_{\text{Lines } L \subseteq \F_q^n}  L, \label{eq:tTq} \\
    \tTqx:=& \sum_{\text{Lines } L \subseteq \F_q^n} \frac{x_i}{q^{n-i}} \ L \qquad \text{for } L = e_i + \sum_{k>i} c_k e_k. \label{eq:tTqx}
\end{align}
In other words, \eqref{eq:tTq} sets $x_L = 1/[n]_q$ for all lines $L$ and \eqref{eq:tTqx} determines the probabilities via $\mathcal{X}_{G/B}(q,\x)$.

The definition of $\mathcal{F}_n(q)$ then immediately implies that the action of $\tTq$ and $\tTqx$ on a complete flag 
$F = (V_0 \subseteq V_1 \subseteq \cdots \subseteq V_n = \F_q^n)$ by left multiplication is given by
\begin{align}
    \label{eq:LRBflags}
        \tTq \cdot  F &= \sum_{\text{lines } L \subseteq \F_q^n}  (V_0 \subseteq L \subseteq L + V_1 \subseteq \cdots \subseteq L+ V_{n-1} 
        \subseteq V_n)^{\widehat{}}, \\
        \tTqx \cdot  F &= \sum_{\text{lines } L \subseteq \F_q^n} \big( (V_0 \subseteq L \subseteq L + V_1 \subseteq \cdots \subseteq L+ V_{n-1} \subseteq V_n)^{\widehat{}}\ \big)\cdot \mathcal{X}_{G/B}(q,\x).  \label{eq:weightedLRBflags}
\end{align}
Note that we have seen the right-hand-side of \eqref{eq:LRBflags} and \eqref{eq:weightedLRBflags} before, in \eqref{eq:r2tflags} 
and \eqref{eq:r2tflagsweighted}, when we claimed that they describe the action of $\hatTq$ and $\hatTqx$, respectively. Since both 
$\hatTq$ and $\hatTqx$ are defined in $\HH_n(q)$, it is not immediately obvious that this is true, but we will now justify this claim:

\begin{prop}\label{prop:flagactions}
    Let $\tTq$ and $\tTqx$ be as in \eqref{eq:LRBflags} and \eqref{eq:weightedLRBflags}, and let $F = (V_0 \subseteq V_1 \subseteq \cdots \subseteq V_n)$ 
    be a flag in $\F_q^n$. Then the following identities hold:
    \begin{enumerate}
        \item {\rm \cite[Prop. 4.10]{qr2r}:}
        One has $\tTq \cdot F = F \cdot \hatTq$.
        \item One has $\tTqx \cdot F = F \cdot \hatTqx$.
    \end{enumerate}
Hence the matrices $\tTqx$ and $\hatTqx$ are transposes of one another.
\end{prop}

\begin{proof}
Recall that $\hatTqx = \hatTq \cdot \mathcal{X}_{G/B}(q,\x)$, where 
$F \cdot \mathcal{X}_{G/B}(q,\x)=  \frac{x_i}{q^{n-i}} \, F$ if $V_1 = e_i + \sum_{j = i+1}^n c_j e_j$ for $c_j \in \F_q$.
Part (1) appears in \cite[Prop 4.10]{qr2r}; comparing the definition of $\mathcal{X}_{G/B}(q,\x)$ with the coefficients of each term on the right-hand-side 
of \eqref{eq:weightedLRBflags} then gives (2). 
\end{proof}

The purpose of introducing $\tTqx$ is that the theory of left regular bands is a powerful tool in understanding the characteristic polynomials 
and stationary distributions of Markov chains arising in this context. Combining \cref{prop:flagactions} with 
\cref{theorem.topcommutativediagram,theorem.bottomcommutativediagram} will allow us to translate these properties to 
$\hatTqx, \Tqx$ and $\wordTqx$, despite them \emph{not} coming directly from a known left regular band. 

\section{Characteristic polynomial of the $q$-Tsetlin library}
\label{section.eigenvalues}

In this section, we compute the minimal and characteristic polynomials of $\hatTqx$ (\cref{ss:eigenvalue flags}) $ \Tqx$ (\cref{ss:eigenvalue perm}) and $\wordTqx$ (\cref{ss:eigenvalue word}). To do so, we first compute 
these polynomials for $\tTqx$ using the theory of left regular bands in \cref{ss:eigenvalue flags}. 

\subsection{Characteristic polynomial for flags}
\label{ss:eigenvalue flags}

Bidigare--Hanlon--Rockmore~\cite{bidigare_hanlon_rockmore.1999} and Brown \cite{Brown} developed a powerful theory for computing 
the characteristic polynomial of a Markov chain $\mathcal{M}$ arising from a left regular band. In particular, we consider $a$ to be an 
element in a left regular band algebra ${\bf k} M$, and define a Markov chain $\mathcal{M}$ generated by the action of $a$ on certain ``top dimensional'' 
elements in ${\bf k} M$ via left-multiplication. Since this perspective and its techniques are explained eloquently 
in~\cite{bidigare_hanlon_rockmore.1999,Brown} in full generality, we restrict ourselves to detailing only what is essential to understand the $q$-Tsetlin library. 

In our context, the main characters are as follows:
\begin{enumerate}
    \item The left-regular-band algebra is $\C \mathcal{F}_n(q)$ from \cref{ss:Brown LRB};
    \item The element $a$ is $\tTqx$ from \eqref{eq:tTqx};
    \item The top dimensional elements are complete flags in $\mathcal{F}_n(q)$, and the action of $\tTqx$ is thus described in \eqref{eq:weightedLRBflags}.
\end{enumerate}

In \cite[Section 5.2]{Brown}, Brown shows that the eigenvalues of the Markov chain generated by $\tTqx$ acting on complete flags are indexed by subspaces 
$V \subseteq \F_q^n$; write this eigenvalue as $\lambda_V$, and its multiplicity as $m_V$. Then
\begin{align}
\label{eq:eigenvaluebasic}    \lambda_V &= \sum_{\text{Lines } L \subseteq V} [L] \big(L \cdot \mathcal{X}_{G/B}(q,\x)\big),\\
 \label{eq:multiplicitybasic}   m_V& = d_{n-\dim(V)}(q),
\end{align}
where $d_n(q)$ are the $q$-derangement numbers in \eqref{eq:qderangements}.

With a bit more work, one can show that with the choice of weights given by $\mathcal{X}_{G/B}$, the eigenvalues of $\tTqx$ are indexed by 
subsets $S \subseteq [n]$. We again use the change-of-coordinates $y_i:= x_i/q^{n-i}$ from \eqref{eq:ydef}.

\begin{prop}\label{prop:flageigenvalues}
The eigenvalues of $\tTqx$ are given by $\lambda_{S}(q,\x)$ for every 
$S = \{i_1 > i_2 > \cdots > i_k\}$, where 
\[  \lambda_S(q,\x) = \sum_{j=1}^k \frac{x_{i_j}}{q^{n-i_j-j+1}}. \]
The multiplicity of $\lambda_S(q,\x)$ is 
\[ m_{G/B}(S)= \begin{cases} d_{n-k}(q) \  q^{(n-i_1) + (n-1-i_2) + \cdots + (n-k+1-i_k)} & k\neq n \\
1 & k=n.
\end{cases}\] 
Moreover, $\tTqx$ is diagonalizable whenever each $x_i$ is real and nonnegative.
\end{prop}

\begin{proof}
Fix $S = \{ i_1 > i_2 > \cdots > i_k \}$ and $e_1, \ldots, e_n$ as the standard basis of $\F_q^n$. 

We claim that the eigenvalue $\lambda_V$ from \eqref{eq:eigenvaluebasic} coincides with $\lambda_S(q,\x)$ whenever $V$ is spanned by elements of the form 
\begin{equation} \label{eq:vdef}
v_{i_j}:= e_{i_j} + \sum_{\substack{\ell>i_j \\ \ell \neq i_1 \neq i_2 \neq \cdots \neq i_{j-1}}} c_\ell  \, e_\ell 
\end{equation}

To see why, first recall that $L \cdot \mathcal{X}_{G/B}(q,\x) = y_i \, L= x_i/q^{n-i}\, L $ for $L = e_i + \sum_{j > i} c_j e_j$. There are $q^{n-i_1}$ 
lines starting with $e_{i_1}$ with weight $y_{i_1}$, $q^{n-1-i_2}$ lines starting with $e_{i_2}$, whose span is linearly independent to the above line, with weight 
$y_{i_2}$, and more generally $q^{n-1-i_j}$ lines starting with $e_{i_j}$. Thus,
\[
\lambda_V = y_{i_1} + q y_{i_2} + \cdots + q^{k-1} y_{i_k} = \lambda_S(q,\x).
\]
To prove the desired multiplicities, by \eqref{eq:multiplicitybasic}, it remains to count the number of subspaces $V$ giving $\lambda_V = \lambda_S(q,\x)$. 
By our above argument, $\lambda_V = \lambda_S$ if and only if $V = \langle v_{i_1}, \ldots, v_{i_k} \rangle$. When $k<n$, there are precisely 
\[ 
	q^{(n-i_1) + (n-1-i_2) + \cdots + (n-k+1-i_k)} 
\] 
such subspaces, corresponding to picking the coefficients $c_\ell$ in \eqref{eq:vdef}. On the other hand, if $k=n$, then there is only one subspace 
spanning $\F_q^n$, and so $m_{G/B}({\F_q^n}) = m_{G/B}([n]) = 1$.

Finally, the diagonalizability of $\tTqx$ follows from basic properties in \cite[Theorem 5, \S 8]{Brown}, since $\x$ is a tuple of 
weights, i.e. nonnegative real numbers.
\end{proof}

Recall that by \cref{prop:flagactions}, the transition matrix $\hatTqx$ is the transpose of the transition matrix for $\tTqx$. 
We thus immediately obtain the following.

\begin{cor}\label{cor:flageigenvalues}
    Let $q$ be a prime power.
    The eigenvalues and multiplicities of $\hatTqx$ are given by $\lambda_S(q,\x)$ and $m_{G/B}(S)$ as in 
    \cref{prop:flageigenvalues}. Moreover, $\hatTqx$ is diagonalizable.
\end{cor}

In particular, this proves \cref{cor.eigenvalues R2T}.

\subsection{Characteristic polynomials for permutations}
\label{ss:eigenvalue perm}

We will now use \cref{prop:flageigenvalues} and \cref{cor:flageigenvalues} to compute the characteristic polynomial of 
$\Tqx$, thereby proving \cref{thm:evalues Sn} from \cref{section.results}. 

Recall from \cref{cor:diagonalizability} that the set $\Lambda_{\symm_n}$ of eigenvalues of $\Tqx$ are contained in the set of eigenvalues 
of $\hatTqx$, which by \cref{cor:flageigenvalues} is given by
\[ 
	\Lambda_{G/B} = \bigg\{ \lambda_S(q,\x) = \sum_{j=1}^k \frac{x_{i_j}}{q^{n-i_j-j+1}} \; \bigg| \; S = \{ i_1 > i_2 > \cdots > i_k \} \subseteq [n] \bigg\}. 
\]
Thus, it remains to compute the multiplicities $m_{\symm_n}(S)$ of $\lambda_S(q,\x)$ for $\Tqx$, which are not implied by the multiplicities for $\hatTqx$. In particular, $\Tqx$ is an $n! \times n!$ matrix, 
whereas $\hatTqx$ is an $[n]!_q \times [n]!_q$ matrix; hence these multiplicities will differ considerably. A priori, it is possible for some 
$\lambda_S(q,\x)$ with $m_{G/B}(S) > 0$ to occur with zero multiplicity in $\Tqx$.

To prove \cref{thm:evalues Sn2}, we use two facts. First, the specialization of $\Tqx$ when $q=1$ is well-understood. In particular,  Donnelly \cite{Donnelly.1991}, Kapoor--Reingold \cite{kapoor1991stochastic}, and Phatarfod \cite{Phatarfod.1991} independently showed that the eigenvalues of $\mathcal{T}(1,\x)$ are of the form $\lambda_S(1,\x) = x_{i_1} + \cdots + x_{i_k}$. They prove that for $S = \{ i_1 > i_2 > \cdots > i_k \}$, the multiplicity of $\lambda_{S}(1,\x)$ is
$d_{n-k}$, where $d_{n-k}$ is the derangement number counting the number of permutations of $\symm_{n-k}$ with no fixed points.

The second fact is the following folklore result, used e.g. in \cite{Lusztig03}\footnote{We thank Darij Grinberg for communicating this result to us.}.

\begin{lemma}\label{darij_lemma}
Suppose $M(q)$ is an $\ell \times \ell$ matrix whose entries are in $\C(q)$ with the following properties:
\begin{enumerate}
    \item For infinitely many values of $q$, we have that $M(q)$ satisfies 
    \[ \prod_{i=1}^{j} (M(q) -\xi_i(q) \, {\sf Id}_{\ell \times \ell}) = 0, \]
    where ${\sf Id}_{\ell \times \ell}$ is the identity matrix, and the
$\xi_i(q)$ are in $ \C(q)$ and pairwise distinct. 
\item When $q=1$, the elements $\xi_i(1)$ are well-defined and pairwise distinct, and the characteristic polynomial of $M(1)$ can be factored as 
\[ \chi(t;1) = \prod_{i=1}^{j} \big( t-\xi_i(1)\big)^{m_i}.\]
\end{enumerate}
Then the characteristic polynomial of $M(q)$ is 
\[ \chi(t;q) = \prod_{i=1}^{j} \big(t-\xi_i(q) \big)^{m_i}  \]
whenever the matrix $M(q)$ and corresponding rational functions $\xi_i(q)$ are well-defined.
\end{lemma}

With \cref{darij_lemma}, we are now ready to prove \cref{thm:evalues Sn2}.

\begin{proof}[Proof of \cref{thm:evalues Sn2}]
We will first show that for a fixed tuple of weights $\x$, the transition matrix $\Tqx$ satisfies the conditions of \cref{darij_lemma}. In particular, 
fix $\x$ to be generic such that (1) $x_i \in \mathbb{R}_{\geqslant 0}$ and $\sum_{i=1}^n x_i =1$, and 
(2) $\lambda_S(1,\x) = \lambda_T(1,\x)$ if and only if $S = T$.

With this choice of $\x$, it is clear that $\Tqx$ has entries in $\C(q)$. Second, by \cref{cor:diagonalizability,cor:flageigenvalues}, 
whenever $q=p^m$ for some prime $p$, we have that $\Tqx$ is diagonalizable and $\Lambda_{\symm_n} \subseteq \Lambda_{G/B}$. Hence, for infinitely many values of $q$, we have that $\Tqx$ satisfies 
\begin{equation}\label{eq:minpolyTqx}
    \prod_{S \subseteq [n]} \big( \Tqx - \lambda_S(q,\x) {\sf Id}_{n! \times n!}  \big) = 0.
\end{equation}
By the genericity of $\x$, the $\lambda_S(q,\x)$ are distinct for distinct subsets $S \subseteq [n]$ when $q$ is generic. Hence $\Tqx$ 
satisfies the first hypothesis of \cref{darij_lemma}.

As noted above, \cite{Donnelly.1991,kapoor1991stochastic, Phatarfod.1991} shows the characteristic polynomial of $\mathcal{T}_{\symm_n}(1,\x)$ is given by 
\[ 
	\chi(t;1,\x) = \prod_{S \subseteq [n]} \big( t -  \lambda_S(1,\x) \big)^{d_{n-|S|}}, 
\]
where $\lambda_S(1,\x) = \sum_{j=1}^{|S|} x_{i_{j}}.$ Again, by genericity of $\x$ we have $\lambda_S(1,\x) \neq \lambda_T(1,\x)$ if $T \neq S$. Thus 
$\Tqx$ satisfies the second hypothesis of \cref{darij_lemma}. It follows that when $q\neq 0$ and generic $\x$, we obtain that the 
characteristic polynomial of $\Tqx$ is 
\begin{equation}\label{eq:charpoly}
	\chi(t;q,\x) = \prod_{S \subseteq [n]} \big( t -  \lambda_S(q,\x)\big)^{d_{n-|S|}}. 
\end{equation}
We now claim that \eqref{eq:charpoly} holds for any choice of $\x$. Our genericity arguments show that \eqref{eq:charpoly} holds on a  Zariski dense subset; but \eqref{eq:charpoly} is a polynomial identity,  and hence must hold for all $\x$. 

The diagonalizability of $\Tqx$ for $q \neq 0$ follows by a similar argument. By the same polynomial identity arguments, we have that $\Tqx$ 
satisfies~\eqref{eq:minpolyTqx} whenever $\lambda_S(q,\x)$ and $\Tqx$ are well-defined, i.e. when $q \neq 0$. Then $\Tqx$ will be diagonalizable if 
the $\lambda_S(q,\x)$ are pairwise distinct, i.e. if $\lambda_S(q,\x) = \lambda_T(q,\x)$ means that $S = T$. 
\end{proof}

\subsection{Characteristic polynomial for words}
\label{ss:eigenvalue word}

We now turn to computing the characteristic polynomial for $\wordTqx$, thereby proving \cref{thm:evalues words}.
Our techniques are similar to those used in \cref{ss:eigenvalue perm}, but we need to take some care translating from $\x$ to $\xbar$ coordinates.

Recall that the upper sets of $P_\m$ are parametrized by weak compositions $\av = (a_1, \dots, a_\ell)$, where $a_i \leqslant m_i$, 
and the integers $n_i$ and sets $M_i$ are defined in \eqref{equation.Midef}.

\begin{proof}[Proof of \cref{thm:evalues words}]
We first prove that the eigenvalues of $\wordTqx$ are of the form $\lambda_{\av}(q,\xbar)$ in \eqref{eq:simplifiedwordevalue}. Recall from 
\cref{cor:diagonalizability} that the set $\Lambda_{W_\m}$ of eigenvalues of $\wordTqx$ is contained in the set $\Lambda_{\symm_n}$
of eigenvalues of $\Tqx$. For a set $S = \{ i_1 > i_2 > \cdots > i_k \}$, recall from the proof of \cref{thm:evalues Sn2} that one can write 
$\lambda_{S}(q,\x) \in \Lambda_{\symm_n}$ as 
\begin{equation} \label{eq:lambdaSiny}
	\lambda_S(q,\x) = \sum_{j=1}^k \frac{x_{i_j}}{q^{n-i_j-j+1}} =\sum_{j=1}^k q^{j-1} y_{i_j}.  
\end{equation}
By \cref{lemma:howtomapybar}, $\proj_{W_\m}(y_{n_{j-1}+i}) = \yo_j$ for $1 \leqslant i \leqslant m_j$. Writing the 
eigenvalue $\lambda_S(q,\x)$ in terms of the $\ybar$ (and hence the $\xbar$, which is our ultimate goal) depends on the set $S \cap M_j$ for each 
$1 \leqslant j \leqslant \ell$. (Note that this is equivalent to computing $\proj_{W_\m}(\lambda_S(q,\x))$.) Write 
\[ 
a_j:= |S \cap M_j|, \qquad 
\text{and}
\qquad b_j:= a_{j+1} + a_{j+2} + \cdots +a_{\ell}. 
\]

Then using the right-hand-side of \eqref{eq:lambdaSiny}, we pair elements appearing in each $M_{1}, \ldots, M_{\ell}$. Note that elements in the 
same $M_j$ appear consecutively in $\sum_{j=1}^k q^{j-1} y_{i_j}$, but differ by a power of $q$. 

Suppose $i_j= \max(S \cap M_j)$; thus for 
a given set of elements $i_{r} > \cdots > i_s \in M_j$, under $\proj_{W_\m}$, one has
\[ 
	q^{r-1} y_{i_r} + \cdots + q^{s-1}y_{i_s}  \mapsto q^{r-1} [a_j]_q \yo_j.
\]
Moreover, $r-1 = b_j$, since $a_{j+1} + \cdots + a_\ell $ terms of $S$ necessarily precede $i_r$. It follows that 
\begin{equation}
    \proj_{W_\m}\big(\lambda_S(q,\x)\big) = \sum_{j=1}^{\ell} q^{b_j}[a_j]_q \yo_j = \sum_{j=1}^{\ell} \frac{q^{b_j}[a_j]_q}{q^{n-n_j}[m_j]_q} \xo_j,
\end{equation}
where the second equality follows from the fact that $\yo_j = \frac{\xo_j}{q^{n-n_j}[m_j]_q}$. This proves that the eigenvalues of $\wordTqx$ are of the form 
\[ \lambda_{\av}(q,\xbar) = \sum_{j=1}^{\ell} \frac{q^{(a_{j+1} + \cdots + a_\ell)}[a_j]_q}{q^{n-n_j}[m_j]_q} \xo_j\]
as desired.

We next prove that the multiplicity of $\lambda_\av$ is $d_{P_\m \setminus S}$ as described in \cref{section.qTsetlin on words}. We again 
use \cref{darij_lemma}: by an analogous argument to the proof of \cref{thm:evalues Sn2}, when $\xbar$ is generic, the conditions of 
\cref{darij_lemma} (1) are satisfied. By \cite[Theorem 5.3]{AyyerKleeSchilling.2014}, the characteristic polynomial of $\mathcal{T}_{W_\m}(1,\x)$ is 
\[ 
	\prod_{S \subseteq [n]} (t - \lambda_S(1,\x))^{d_{P_\m \setminus S}}. 
\]
By our above argument, $\lambda_S(1,\x) = \lambda_{\av}(1,\xbar)$ with $\av= (|S \cap M_1|, |S \cap M_2|, \ldots, |S \cap M_\ell|).$ 
Hence (using \cref{darij_lemma} and an analogous argument about the genericity of the $\xbar$ as in the proof of \cref{thm:evalues Sn2}) 
we may conclude that
\[ 
	\prod_{S \subseteq [n]} (t - \lambda_{\av}(q,\xbar))^{d_{P_\m \setminus S}} 
\]
is the characteristic polynomial of $\wordTqx$ for all $q$ and all choices of $\xbar$ for which $\wordTqx$ is defined.
\end{proof}

\begin{example}
Suppose $\m = (2,3,1)$ and consider $S = \{ 1,2, 3,5 \}$. Then $a_1 = 2$, $a_2 = 2$ and $a_3 = 0$, and $n_1 = 2$, $n_2 = 5$ and $n_3 = 6= n$. Thus
\begin{align*}
    \lambda_{\{ 1,2,3, 5 \}}(q,\x) &= \frac{1}{q} x_5 + \frac{1}{q^2} x_3 + \frac{1}{q^2} x_2 + \frac{1}{q^2} x_1 
    \\ &= y_5 + q \,y_3 + q^2 \,y_2 + q^3 \, y_1  \\
    &= \yo_2 + q \, \yo_2 + q^2 \, \yo_1 +  q^3 \, \yo_1 = [2]_q \yo_2 + q^2[2]_q \yo_1 \\
    &= \frac{[2]_q}{q^{6-5}[3]_q} \xo_2 + \frac{q^2[2]_q}{q^{6-2}[2]_q} \xo_1 = \lambda_{(2,2,0)}(q,\xbar).
\end{align*}
Note that $a_3 = 0$ explains the $q^0$ exponent in the numerator of the $\xo_2$ term, and $a_2 + a_3 = 2$ explains the exponent in the 
numerator of the $\xo_1$ term.
\end{example}

\section{The stationary distribution}
\label{section.stationary distribution}

This section is devoted to the proofs of 
\cref{th:eigenvectorH,th:steady state words,theorem.sd flags},
which state the stationary distributions for the various Markov chains.
The proof of \cref{theorem.sd flags} for the stationary distribution for flags uses methods for left regular bands (or more generally 
$\mathscr{R}$-trivial monoids) which are reviewed in \cref{ss:semigroup sd}. The proof itself is given in \cref{ss:flags sd}.
The proofs of \cref{th:eigenvectorH,th:steady state words} of the stationary distributions for the $q$-Tsetlin library on 
permutations and words are given in \cref{ss:ss permutations,ss:ss words}, respectively, using lumping
established in \cref{cor:lumping}.

\subsection{Computing stationary distributions using semigroup theory}
\label{ss:semigroup sd}

As shown in \cref{ss:Brown LRB}, our $q$-Tsetlin library on flags and cosets introduced in \cref{section.qTsetlin on flags}
is equivalent to a Markov chain arising from a left regular band~\cite{Brown,BraunerComminsReiner}. Hence we are able to utilize the methods developed 
in~\cite{ASST.2015,RhodesSchilling.2019} to compute the stationary distribution, which we now briefly review.

The methods developed in~\cite{RhodesSchilling.2019} work in the general setting for Markov chains described by $(S,X)$, where
$S$ is a finite semigroup $S$ with a generating set $X=\{X_1,\ldots,X_\ell\}$ and each generator $X_i$ is associated with a probability $x_i$. 
One considers the right Cayley graph $\mathsf{RCay}(S,X)$ with root $\mathbbm{1}$ and a directed edge labeled $X_i$ from $s\in S$
to $s'\in S$ if $s'=s X_i$. In addition, there is an edge labeled $X_i$ from $\mathbbm{1}$ to $X_i$ for $1\leqslant i \leqslant \ell$.
We describe the methods in the case when $S$ is $\mathscr{R}$-trivial, meaning that each edge $s \stackrel{X_i}{\longrightarrow} s'$ 
in $\mathsf{RCay}(S,X)$ either
\begin{enumerate}
\item \textit{stabilizes} $s$, meaning $s'=s$; or
\item is a \textit{transition edge}, meaning there is no directed path from $s'$ back to $s$ in $\mathsf{RCay}(S,X)$.
\end{enumerate}

\begin{remark}
Left regular bands are $\mathscr{R}$-trivial semigroups.
\end{remark}
The \textit{minimal ideal} $\mathcal{I}$ of the right Cayley graph consists of all vertices that have no outgoing transition edges. 
For $\mathscr{R}$-trivial semigroups, the minimal ideal $\mathcal{I}$ hence consists of all vertices which are stabilized by all
generators $X_i$ for $1\leqslant i \leqslant \ell$.

The states of the Markov chain $\mathcal{M}(S,X)$ defined by $(S,X)$ are the elements of the minimal ideal $\mathcal{I}$ and transitions
between states is given by left multiplication by the generators $X_i \in X$. More precisely, in the Markov chain there is a transition
from $s\in \mathcal{I}$ to $s' \in \mathcal{I}$ with probability $x_i$ if $s' = X_i s$.

The stationary distribution $\Psi = (\Psi_s)_{s\in \mathcal{I}}$ has components $\Psi_s$ indexed by the states of the Markov chain $s\in \mathcal{I}$.
As shown in~\cite{ASST.2015,RhodesSchilling.2019}, the component $\Psi_s$ can be computed as follows. Consider the set $P_s$ of all paths in 
$\mathsf{RCay}(S,X)$ from $\mathbbm{1}$ to $s$ consisting only of transition edges. A path in $P_s$ is given by a tuple 
$(X_{i_1},\ldots,X_{i_k})$, where each $X_{i_j}$ corresponds to the label of a transition edge and
\[
	s = X_{i_1} \cdots X_{i_k}.
\]
Denote by $\mathsf{Stab}_j$ the set of all indices $i$ such that $X_i$ stabilizes $X_{i_1} \cdots X_{i_j}$, that is,
\[
	X_{i_1} \cdots X_{i_j} \cdot X_i = X_{i_1} \cdots X_{i_j}.
\]
The probability associated to the path $(X_{i_1},\ldots,X_{i_k})$ is
\begin{equation}
\label{equation.prob p}
	\frac{x_{i_1} \cdots x_{i_k}}{\prod_{j=1}^{k-1} (1-\sum_{i \in \mathsf{Stab}_j} x_i)}.
\end{equation}
Then $\Psi_s$ is the sum over all paths in $P_s$ over their associated probability as in~\eqref{equation.prob p}.

\subsection{Stationary distribution for flags}
\label{ss:flags sd}
In this section, we prove \cref{theorem.sd flags} using the methods described in \cref{ss:semigroup sd}.

The semigroup for the $q$-Tsetlin library on flags is the $q$-free left regular band $\mathcal{F}_n(q)$ introduced in \cref{ss:Brown LRB}.
The generators of the left regular band are the $[n]_q$ lines in $\mathbb{F}_q^n$. Recall that the line $L = \langle e_i + \sum_{k=i+1}^n c_k e_k\rangle$ with
$c_k \in \mathbb{F}_q$ is assigned the probability 
\begin{equation}
y_i=\frac{x_i}{q^{n-i}},
\end{equation} 
where $0 < x_i \leqslant 1$ and $x_1+\cdots+x_n=1$. Note
that there are $q^{n-i}$ lines $L$ with leading term $e_i$ since there are $q$ choices for each coefficient $c_k$ for $i<k\leqslant n$.

\begin{example}
Let $q=2$ and $n=3$. Part of the right Cayley graph of $\mathcal{F}_3(2)$ is depicted in \cref{figure.partial RCay}.
Consider the flag
\[
	F_{132}= (\emptyset \subseteq \langle e_{1} \rangle \subseteq \langle e_{1}, e_{3} \rangle \subseteq \langle e_1, e_2, e_3 \rangle).
\]
The probability to reach $F_{132}$ is
\[
	\Psi(q,\x)_{F_{132}} = \frac{y_1(y_3+y_1)(2y_1+2y_2)}{(1-y_1)(1-y_3-2y_1)} = \frac{y_1(y_3+y_1)}{1-y_1},
\]
where the numerator comes from the transition edges and the denominator comes from the stabilizers. Note that $x_1=4y_1$, $x_2=2y_2$, and
$x_3=y_3$. Since $x_1+x_2+x_3=1$, the terms $2y_1+2y_2$ and $1-y_3-2y_1$ are equal and cancel. This expression agrees with that
of $\Psi(q,\x)_{F_{132}}$ in \cref{example.q=2 n=3 flag}.
\end{example}

\begin{figure}
\begin{center}
    \begin{tikzpicture}[auto,scale=0.7]
\node (I) at (0, 0) {$\emptyset$};
\node (A1) at (-4,-1.5) {$\langle e_1\rangle$};
\node (A2) at (-0.5,-1.5) {$\langle e_1+e_2\rangle$};
\node (A3) at (2,-1.5) {$\langle e_1+e_3\rangle$};
\node (A4) at (5,-1) {};
\node(B1) at (-8.5,-4) {$\langle e_1,e_2\rangle$};
\node(B2) at (-5,-4) {$\langle e_1,e_2+e_3\rangle$};
\node(B3) at (0.5,-4) {$\langle e_1, e_3\rangle$};
\node(C1) at (-9,-6) {$\langle e_1,e_2,e_3\rangle$};
\node(C3) at (0.5,-6) {$\langle e_1,e_3,e_2\rangle$};

\draw[edge,blue,thick] (I) -- (A1) node[midway, left] {\scriptsize $\langle 1\rangle$\;\;};
\draw[edge,blue,thick] (I) -- (A2);
\draw[edge,blue,thick] (I) -- (A3);
\draw[edge,blue,thick] (I) -- (A4) node[midway, right] {\;\;\;\;\scriptsize $[n]_q=7$ lines};

\draw[edge,blue,thick] (A1) -- (B1) node[midway, left] {\scriptsize $\langle 2\rangle, \langle 1+2\rangle$\;\;};
\draw[edge,blue,thick] (A1) -- (B2) node[midway, right] {\scriptsize $\begin{array}{l} \langle 2+3\rangle, \\ \langle 1+2+3\rangle \end{array}$};
\draw[edge,blue,thick] (A1) -- (B3) node[midway, right] {\;\;\scriptsize $\langle 3\rangle, \langle 1+3\rangle$};
\draw[edge,blue,thick] (B1) -- (C1) node[midway, right] {\scriptsize $q^2$ lines};
\draw[edge,blue,thick] (B2) -- (C1) node[midway, right] {\; \scriptsize $q^2$ lines};
\draw[edge,blue,thick] (B3) -- (C3) node[midway, right] {\;\;\scriptsize $\begin{array}{l}  \langle 2\rangle, \langle 2+3\rangle, \\ \langle 1+2\rangle, \langle 1+2+3\rangle \end{array}$};

\path
(A1) edge [loop above, red] node[left] {\scriptsize  $\langle 1\rangle$} (A1)
(B1) edge [loop above, red] node[left] {\scriptsize $\begin{array}{l} \langle 2\rangle, \langle 1\rangle, \\ \langle 1+2\rangle \end{array}$} (B1)
(B2) edge [->,loop, out=300,in=270,looseness=6, red] node[right] {\scriptsize $\langle 1\rangle, \langle 2\rangle, \langle 2+3\rangle$} (B2)
(B3) edge [loop above, left,red] node[right] {\scriptsize $\langle 3\rangle, \langle 1\rangle, \langle 1+3\rangle$} (B3);
\end{tikzpicture}
\end{center}
\caption{Partial right Cayley graph for the case $q=2$ and $n=3$. On the colored arrows we abbreviate $\langle e_i\rangle$ by $\langle i\rangle$.
\label{figure.partial RCay}}
\end{figure}

\begin{proof}[Proof of \cref{theorem.sd flags}]
We will prove the theorem under the assumption that $x_1+\cdots + x_n=1$. In particular, we first show that there exists a stationary distribution satisfying
\begin{equation}
\label{equation.for proof}
\Psi(q,\x)_{F_\pi} = \frac{ \prod_{k=1}^n f(\pi_1,\ldots, \pi_k)}
{\prod_{k=1}^{n-1}\left( 1 - \sum_{s\in \{\pi_1,\ldots,\pi_k\}} \frac{x_s}{q^{n-s-b_{k+1}(s)}}\right)}
\end{equation}
with 
\begin{equation}
\label{equation.f pi}
	f(\pi_1,\ldots,\pi_k) = \sum_{\substack{s\in \{\pi_1,\ldots,\pi_k\}\\ s\leqslant \pi_k}} \frac{x_s}{q^{n-s-b_k(s)}} (q-1)^{\chi(s<\pi_k)},
\end{equation}
and $b_k(s) = \# \{ t\in \{\pi_1,\ldots,\pi_{k-1}\} \mid t>s \}$, and $\chi(s<\pi_k)=1$ if $s<\pi_k$ and 0 otherwise. To obtain the formulas stated in \cref{theorem.sd flags}, one may replace $x_i$ by $\frac{x_i}{x_1+\cdots+x_n}$ and multiply~\eqref{equation.for proof}
by $\frac{1}{x_1+\cdots+x_n}$. Note that the factor $k=1$ in the denominator in \cref{theorem.sd flags} is equal to 1.

By \cref{cor:lumping}, we have $\Psi(q,\x)_F = \Psi(q,\x)_{F'}$ if $\proj_{\symm_n}(F) = \proj_{\symm_n}(F')$. Hence, for a given 
permutation $\pi = \pi_1 \,\pi_2\,\ldots\,\pi_n$, it suffices to compute the stationary distribution for the flag 
\[ F_{\pi}:= (\emptyset \subseteq \langle e_{\pi_1} \rangle \subseteq \langle e_{\pi_1}, e_{\pi_2} \rangle \subseteq \cdots \subseteq V_n). \]

To reach $F_{\pi}$ from the root of the right Cayley graph, in the first step we need to pick line $\langle e_{\pi_1} \rangle$ with probability
$\frac{x_{\pi_1}}{q^{n-\pi_1}}$. This is the term $f(\pi_1)$ in~\eqref{equation.f pi}, which is the factor $k=1$ in the numerator 
in~\eqref{equation.for proof}. The stabilizer of the line $\langle e_{\pi_1} \rangle$ is only the line itself. Since its associated
probability is $\frac{x_{\pi_1}}{q^{n-\pi_1}}$, it contributes the terms $\frac{1}{1-\frac{x_{\pi_1}}{q^{n-\pi_1}}}$ to the stationary distribution,
which is the term $k=1$ in the denominator of~\eqref{equation.for proof}.

Next, we need to look at all generators (i.e. lines), which take us from $\langle e_{\pi_1} \rangle$ to $\langle e_{\pi_1}, e_{\pi_2} \rangle$. We need to
distinguish the cases when $\pi_1<\pi_2$ and $\pi_1>\pi_2$. First consider the case $\pi_1<\pi_2$. The $q-1$ lines 
$\langle e_{\pi_1} + c e_{\pi_2} \rangle$ with $c\neq 0$ and the line $\langle e_{\pi_2} \rangle$ go from $\langle e_{\pi_1} \rangle$ to 
$\langle e_{\pi_1}, e_{\pi_2} \rangle$. This happens with probability 
\[ (q-1) \frac{x_{\pi_1}}{q^{n-\pi_1}} + \frac{x_{\pi_2}}{q^{n-\pi_2}} = f(\pi_1, \pi_2), \]
which is the factor $k=2$ in the numerator of~\eqref{equation.for proof}. Next consider $\pi_1>\pi_2$. The $q$ lines $e_{\pi_2} + c e_{\pi_1}$ for 
any $c\in \mathbb{F}_q$ take us from $\langle e_{\pi_1} \rangle$ to $\langle e_{\pi_1}, e_{\pi_2} \rangle$. This happens with the probability 
$q \frac{x_{\pi_2}}{q^{n-\pi_2}} = f(\pi_1,\pi_2)$, which is the factor $k=2$ in the numerator of~\eqref{equation.for proof}.

If $\pi_1<\pi_2$, the lines $\langle e_{\pi_1} + c e_{\pi_2} \rangle$ for $c\in \mathbb{F}_q$ and $\langle e_{\pi_2} \rangle$ stabilize the subspace
$\langle e_{\pi_1}, e_{\pi_2} \rangle$ with probability $p=q \frac{x_{\pi_1}}{q^{n-\pi_1}} + \frac{x_{\pi_2}}{q^{n-\pi_2}}$. This contributes 
$\frac{1}{1-p}$ to the stationary distribution, which is the term $k=2$ in the denominator of~\eqref{equation.for proof}. 
If $\pi_1>\pi_2$, the lines $\langle e_{\pi_2} + c e_{\pi_1} \rangle$ for $c\in \mathbb{F}_q$ and $\langle e_{\pi_1} \rangle$ stabilize the subspace
$\langle e_{\pi_1}, e_{\pi_2} \rangle$ with probability $p=\frac{x_{\pi_1}}{q^{n-\pi_1}} + q\frac{x_{\pi_2}}{q^{n-\pi_2}}$. This contributes
$\frac{1}{1-p}$ to the stationary distribution, which is the term $k=2$ in the denominator of~\eqref{equation.for proof}. 

In general, to get from the subspace $\langle e_{\pi_1},\ldots,e_{\pi_{k-1}} \rangle$ to the subspace $\langle e_{\pi_1},\ldots, e_{\pi_{k}} \rangle$,
the lines 
\[ \langle e_{\pi_k} + \sum_{t\in \{\pi_1,\ldots, \pi_{k-1}\}, t>\pi_k} c_t e_t \rangle\] with $c_t \in \mathbb{F}_q$ and the lines 
\[ \langle e_s +c e_{\pi_k}+ \sum_{t\in \{\pi_1,\ldots,\pi_{k-1}\}, t>s} c_s e_t \rangle\] with $s\in \{\pi_1,\ldots, \allowbreak \pi_{k-1}\}$ and $s<\pi_k$, $c\neq 0$, 
$c_t \in \mathbb{F}_q$ can be added. The first set of lines comes with probability $q^{b_k(\pi_k)} \frac{x_{\pi_k}}{q^{n-\pi_k}}$,
whereas the second set of lines comes with probabilities 
\[ \sum_{s\in \{\pi_1,\ldots,\pi_k\}, s< \pi_k}\frac{x_s}{q^{n-s-b_k(s)}} (q-1) \]
for a total of $f(\pi_1,\ldots, \pi_k)$, which is the $k$-th factor in the numerator of $\Psi(q,\x)_{F_\pi}$.

The stabilizer of the subspace $\langle e_{\pi_1},\ldots,e_{\pi_k} \rangle$ are the lines
$\langle e_s + \sum_{t\in \{\pi_1,\ldots,\pi_k\}, t>s} c_t e_t \rangle$ for $c_t \in \mathbb{F}_q$ and $s\in \{\pi_1,\ldots,\pi_k\}$.
The probability for picking these lines is 
\[ p=\sum_{s\in \{\pi_1,\ldots,\pi_k\}} \frac{x_s}{q^{n-s-b_{k+1}(s)}}.\] The contribution is a factor
$\frac{1}{1-p}$ to the stationary distribution which is the $k$-th factor in the denominator in~\eqref{equation.for proof}.
This completes the proof of the existence of the stationary distribution.
It is clear from this calculation that each $\Psi(q,\x)_F$ is a probability, and hence the sum is $1$.

We will now show that the $q$-Tsetlin library on cosets is irreducible.
Consider the long word in $\symm_n$ given by $w_0 = n\; n-1\; \dots\;1$. It is easy to see that there is a unique double coset, 
$[w_0] = B^- w_0 B$, that corresponds to it. 
We have shown in \cref{theorem.topcommutativediagram} that $\proj_{\symm_n}$ projects the $q$-Tsetlin library on cosets to that on permutations.
Since the latter is irreducible, $w_0$ must occur starting from any configuration
in the $q$-Tsetlin library on permutations. Thus, $[w_0]$ must occur starting from any configuration in the $q$-Tsetlin library on cosets.
We have shown above that every coset gets assigned a positive probability in some stationary distribution. Hence, every coset is recurrent, proving irreducibility 
and uniqueness.
\end{proof}

\begin{remark}
\label{rem:last factor}
Following the above proof, note that for the subspace $\langle e_{\pi_1},\ldots,e_{\pi_{n-1}} \rangle$ any line either 
stabilizes the subspace with some probability $p$ or falls into the ideal (that is, reaching $F_{\pi}$) with probability $1-p$. 
Since the stabilizer contributes the factor $\frac{1}{1-p}$, the two cancel out. This shows that the last factor (that is, $k=n$)
for $\Psi(q,\x)_{F_\pi}$ in \cref{theorem.sd flags} is actually $1$. 
\end{remark}

The next corollary gives a strengthening of \cref{cor:lumping} for $\m$-compatible weights, which will be used in \cref{lemma.x compatible} and
the proof of the stationary distribution for words.

\begin{cor}
\label{cor:y equal}
Let $\pi \in \symm_n$ and $\pi'= s_i \cdot \pi$ be obtained from $\pi$ by interchanging $i$ and $i+1$ for some $1\leqslant i<n$. If $y_i=y_{i+1}$, we have
\[
	\Psi(q,\x)_{F_\pi} = \Psi(q,\x)_{F_{\pi'}}.
\]
\end{cor}

\begin{proof}
Rewriting $\Psi(q,\x)_{F_\pi}$ of \cref{theorem.sd flags} in terms of $y_i$ instead of $x_i$, we obtain
\[
	\Psi(q,\x)_{F_\pi} = \prod_{k=1}^n \frac{f(\pi_1,\ldots, \pi_k)}{1-\sum_{s\in \{\pi_1,\ldots,\pi_{k-1}\}} y_s q^{b_k(s)}},
\]
where
\[
	f(\pi_1,\ldots ,\pi_k) = \ \sum_{\substack{s\in \{\pi_1,\ldots,\pi_k\}\\ s\leqslant \pi_k}} y_s q^{b_k(s)} (q-1)^{\chi(s<\pi_k)}.
\]

Now suppose that $\pi_\ell=i$ and $\pi_{\ell'}=i+1$. Without loss of generality we may assume that $\ell<\ell'$ since otherwise we interchange
$\pi$ and $\pi'$. 
Note that all factors $1 \leqslant k<\ell$ in $\Psi(q,\x)_{F_\pi}$ and $\Psi(q,\x)_{F_{\pi'}}$ are the same. 

Given that $y_i=y_{i+1}$, the factors $\ell\leqslant k<\ell'$ are also the same in $\Psi(q,\x)_{F_\pi}$ and $\Psi(q,\x)_{F_{\pi'}}$ since the inequalities 
$s\leqslant \pi_k$ (resp. $s<\pi_k)$ in 
$f(\pi_1,\ldots, \pi_k)$ (resp. the exponent of $q-1$), and $t>s$ in $b_k(s)$ are satisfied by the same elements since only one of $i$ and $i+1$ appear
in the first $k$ entries of $\pi$ (resp. $\pi'$). 

For the factor $k=\ell'$, the denominators in $\Psi(q,\x)_{F_\pi}$ and $\Psi(q,\x)_{F_{\pi'}}$ are the same since again in the set $\{\pi_1,\ldots,\pi_{\ell'-1}\}$
(resp. $\{\pi'_1,\ldots,\pi'_{\ell'-1}\}$) only $i$ (resp. $i+1$), but not $i+1$ (resp. $i$) is present, and $y_i=y_{i+1}$. For the numerator, we are summing
over the set $\{\pi_1,\ldots,\pi_{\ell'}\}=\{\pi_1',\ldots,\pi'_{\ell'}\}$ with the extra constraints $s\leqslant \pi_{\ell'}=i+1$ in $\Psi(q,\x)_{F_\pi}$
(resp. $s \leqslant \pi'_{\ell'}=i$ in $\Psi(q,\x)_{F_{\pi'}}$). Hence the summands in the numerator of the factor $k=\ell'$ involving $y_s$ with $s\neq i,i+1$ are the 
same in $\Psi(q,\x)_{F_\pi}$ and $\Psi(q,\x)_{F_{\pi'}}$. The summands $s=i,i+1$ in the numerator of the factor $k=\ell'$ involving $y_i$ and $y_{i+1}$ are
\[
\begin{aligned}
	&y_i q^{\#\{t\in \{\pi_1,\ldots,i,\ldots,\pi_{\ell'-1}\} \mid t>i\}}(q-1) + y_{i+1} q^{\#\{t\in \{\pi_1,\ldots,i,\ldots,\pi_{\ell'-1}\} \mid t>i+1\}} 
	&& \text{for $\Psi(q,\x)_{F_\pi}$,}\\	
	&y_i q^{\#\{t\in \{\pi_1,\ldots,i+1,\ldots,\pi_{\ell'-1}\} \mid t>i\}} 
	&& \text{for $\Psi(q,\x)_{F_{\pi'}}$.}
\end{aligned}
\]
Note that 
\[
\begin{split}
	\#\{t\in \{\pi_1,\ldots,i+1,\ldots,\pi_{\ell'-1}\} \mid t>i\} &= \#\{t\in \{\pi_1,\ldots,i,\ldots,\pi_{\ell'-1}\} \mid t>i\} +1,\\
	\#\{t\in \{\pi_1,\ldots,i,\ldots,\pi_{\ell'-1}\} \mid t>i\} &= \#\{t\in \{\pi_1,\ldots,i,\ldots,\pi_{\ell'-1}\} \mid t>i+1\}.
\end{split}
\]
Hence the last two terms in $\Psi(q,\x)_{F_\pi}$ cancel when $y_i=y_{i+1}$ and the numerator for the factor $k=\ell'$ in $\Psi(q,\x)_{F_\pi}$ and 
$\Psi(q,\x)_{F_{\pi'}}$ agree.

For the factors $k>\ell'$, the factors in $\Psi(q,\x)_{F_\pi}$ and $\Psi(q,\x)_{F_{\pi'}}$ are the same.
\end{proof}

\subsection{Stationary distribution for permutations}
\label{ss:ss permutations}

In this section, we prove \cref{th:eigenvectorH}. By \cref{cor:lumping}, we have
\begin{equation}
\label{equation.Psi relation}
	 \Psi(q,\x)_{\pi}  = q^{\coinv(\pi)} \Psi(q,\x)_{F_\pi}
\end{equation}
for $q$ a prime power. Here $\Psi(q,\x)_{F_\pi}$ is given by the expression in \cref{theorem.sd flags}.

\begin{lemma}
\label{lemma.q prime power}
Denote by $\zeta(q,\x)_\pi$ the expression in the right hand side of~\eqref{eq:ssH}. Then
\[
	\Psi(q,\x)_{\pi} = \zeta(q,\x)_\pi
\]
for $q$ a prime power.
\end{lemma}

\begin{proof}
By \cref{rem:last factor}, both $\zeta(q,\x)_\pi$ and $\Psi(q,\x)_{F_\pi}$ have $n-1$ factors. 
We begin by proving that the denominators match term-by-term. 
The $k$-th factor in the denominator for $\zeta(q,\x)_\pi$ and $\Psi(q,\x)_{F_\pi}$ are
\[
\kappa(w_0;\x)-q^{k-n-1} \kappa(\pi_1, \dots, \pi_{k-1};\x)
\]
and
\[
x_1 + \cdots + x_n - \sum_{s\in \{\pi_1,\ldots,\pi_{k-1}\}} \frac{x_s}{q^{n-s-b_{k}(s)}}
=
x_1 + \cdots + x_n - \sum_{i=1}^{k-1} \frac{x_{\pi_i}}{q^{n-\pi_i- \{j \in [k-1] \mid \pi_j > \pi_i\}}},
\]
respectively. By~\eqref{equation.kappa omega0}, $\kappa(w_0;\x)=x_1+\cdots+x_n$.
Now, let $\sigma = (\pi_1, \dots, \pi_{k-1})_>$. Then by~\eqref{eq:kappadef}
\[
\frac{\kappa(\pi_1, \dots, \pi_{k-1};\x)}{q^{n-k+1}} = 
\sum_{i=1}^{k-1} x_{\sigma_i} q^{i+\sigma_i-k-(n-k+1)}
= \sum_{i=1}^{k-1} \frac{x_{\sigma_i}}{q^{n+1-i-\sigma_i}}.
\]
Suppose $\pi_\ell = \sigma_i$. Then that means that there are $i-1$ numbers larger than $\pi_\ell$ in $\{\pi_1, \dots,
\pi_{k-1}\}$. Thus the summand on the right hand side of the above expression is equal to
\[
\frac{x_{\pi_\ell}}{q^{n - \pi_\ell - (i-1)}},
\]
which matches the $\ell$-th term in the denominator of $\Psi(q,\x)_{F_\pi}$. We have thus shown that the denominators are equal.

We now show that the numerators also match term-by-term in $\zeta(q,\x)_\pi$ and $\Psi(q,\x)_{F_\pi}$ up to a power of $q$. 
There are two cases. First, suppose $k \in \text{LRM}(\pi)$. The numerator of the $k$-th term in $\zeta(q,\x)_\pi$ is just  $\kappa(\pi_k;\x)
= x_{\pi_k} q^{\pi_k -1}$. The corresponding numerator in
$\Psi(q,\x)_{F_\pi}$ is $f(\pi_1, \dots, \pi_k) = x_{\pi_k} q^{-n + \pi_k + k - 1}$, since $b_k(\pi_k) = k-1$. Therefore, these factors match up to a factor of $q^{n-k}$. 

Finally, consider the $k$-th term in the numerators, where $k \notin \text{LRM}(\pi)$. Recall that $p_k$ is the notation for the leftmost position 
in $\pi$ whose value is smaller than that of $\pi_k$. Denote $(\pi_{p_k}, \dots, \pi_k)_> = \sigma$ and  $(\pi_{p_k}, \dots, \pi_{k-1})_> = \tau$.
As before, suppose $\pi_\ell = \sigma_i$.
First, look at $f(\pi_1, \dots, \pi_k)$ in the numerator of $\Psi(q,\x)_{F_\pi}$. Note that the only positions that will contribute to the summand 
are those between $p_k$ and $k$, which is a good sanity check.
If $\pi_\ell > \pi_k$, it will not appear in $f(\pi_1, \dots, \pi_k)$. If $\pi_\ell < \pi_k$, it will contribute a factor of
\begin{equation}
\label{eq: l lesser k}
\frac{x_{\pi_\ell} (q-1)}{q^{n-\pi_\ell- \{j \in [k-1] \mid \pi_j > \pi_\ell\}}}.
\end{equation}
Otherwise, $\pi_\ell = \pi_k$, which means $\ell = k$, and we get a factor of
\begin{equation}
\label{eq: l equal k}
\frac{x_{\pi_k}}{q^{n-\pi_k- \{j \in [k-1] \mid \pi_j > \pi_k\}}}.
\end{equation}

Suppose $a$ elements in $\{\pi_1, \dots, \pi_k\}$ are larger than $\pi_k$. Then
\[
	\sigma = (\sigma_1, \dots, \sigma_a, \pi_k, \allowbreak \dots) \quad \text{and} \quad
	\tau = (\sigma_1, \dots, \sigma_a, \widehat{\pi_k}, \dots),
\]
where the hat denotes the absence of the number.
In the $k$-th term of $\zeta(q,\x)_\pi$, we therefore obtain
\[
\kappa(\pi_{p_k}, \dots, \pi_k;\x) 
- q^{-1} \kappa(\pi_{p_k}, \dots, \pi_{k-1};\x) =
\sum_{b = 1}^k x_{\sigma_b} q^{-(k - p_k + 2 - b - \sigma_b)}
 - \sum_{b = 1}^{k-1} x_{\tau_b} q^{-(k - p_k + 1 - b - \tau_b) - 1}.
\]
But $\sigma_b = \tau_b$ for $1 \leqslant b \leqslant a$ and the power of $q$ also matches exactly. Therefore $x_{\pi_\ell}$ does not appear in 
$\zeta(q,\x)_\pi$ if $\pi_\ell > \pi_k$.
If $\ell = k$ and thus $\pi_\ell = \pi_k$, then it does not appear in $\tau$ and so we get a contribution of
\[
x_{\pi_k} q^{-k + p_k - 1 + a + \pi_k},
\]
which matches \eqref{eq: l equal k} again up to a factor of $q^{n-k}$. Lastly, if $\pi_\ell < \pi_k$ and we have assumed $\pi_\ell = \sigma_i$, then $\pi_\ell = \tau_{i-1}$. The contribution of this term is
\[
x_{\pi_\ell} q^{-k + p_k - 2 + i + \pi_\ell}
- x_{\pi_\ell} q^{-k + p_k - 2 + (i-1) + \pi_\ell}
=  x_{\pi_\ell} (q-1) q^{-k + p_k - 3 + i + \pi_\ell}.
\]
By comparison, the factor in \eqref{eq: l lesser k} is
\[
 x_{\pi_\ell} (q-1) q^{-n + \pi_\ell + (i - 2 + p_k - 1)},
\]
because we cannot include $\pi_k$. Thus, this term also matches up to the same overall factor of $q^{n-k}$.

In summary, so far we have proven that
\[
	\zeta(q,\x)_\pi = q^{\sum_{k=1}^{n-1}(n-k) - \inv(\pi)} \Psi(q,\x)_{F_\pi} = q^{\binom{n}{2} - \inv(\pi)} \Psi(q,\x)_{F_\pi}.
\]
The lemma now follows from~\eqref{equation.Psi relation} by observing that $\inv(\pi) + \coinv(\pi) = \binom{n}{2}$, namely
\[
	 \Psi(q,\x)_{\pi}  = q^{\coinv(\pi)} \Psi(q,\x)_{F_\pi} = q^{\coinv(\pi) + \inv(\pi) - \binom{n}{2}} \zeta(q,\x)_\pi = \zeta(q,\x)_\pi.
\]
\end{proof}

Next we extend our results to all $q$.

\begin{lemma}
\label{lemma.all q}
Denote by $\zeta(q,\x)_\pi$ the expression in the right hand side of~\eqref{eq:ssH}. Then 
\[
	\Psi(q,\x)_{\pi} = \zeta(q,\x)_\pi
\]
for all $q$ with $0<q^{-1}\leqslant 1$. 
\end{lemma}

\begin{proof}
The expression of $\zeta(q,\x)_\pi$ as the right hand side of~\eqref{eq:ssH} is rational in $q$ and hence also in $q^{-1}$. This also follows from the
Markov chain tree theorem~\cite{leighton-rivest-1983,anantharam-tsoucas-1989}. Since \cref{lemma.q prime power} establishes
the result for infinitely many specializations of $q$, this proves the claim.
\end{proof}

\begin{proof}[Proof of \cref{th:eigenvectorH}]
The proof of the first part follows directly from \cref{lemma.all q}.

For the second part, we first consider the denominator. The $k$-th factor is 
\[
\kappa(w_0;\x)-q^{k-n-1} \kappa(\pi_1, \dots, \pi_{k-1};\x) 
= x_1 + \cdots + x_n -q^{k-n-1} \kappa(\pi_1, \dots, \pi_{k-1};\x),
\]
by \eqref{equation.kappa omega0}.
Let $\pi_> = (\pi_1, \dots, \pi_{k-1})_>$ for simplicity of notation.
Let us compare the coefficient of $x_i$ for $i \in [n]$. 
If $i$ does not appear in $\pi_>$, then the coefficient is clearly $1$.
Now consider the coefficient of $x_j$ when $j = (\pi_>)_1$, which is 
$1 - q^{k-n-1} \times q^{1 + j - k - 1}$ by \eqref{eq:kappadef}. 
Simplifying the power of $q$, we get $1 + j - n - 2 \leqslant -1$ and hence the coefficient is positive.

Now, consider the numerator. The nontrivial factors are 
\[
\kappa(\pi_{p_k}, \dots, \pi_k;\x) 
- q^{-1} \kappa(\pi_{p_k}, \dots, \pi_{k-1};\x)
\]
whenever $k \notin \LRM(w)$ and $\pi_{p_k} < \pi_k$. Again, we need to compare the coefficient of $x_{\pi_i}$ for $i \in [p_k, k-1]$. Notice that
\[
(\pi_{p_k}, \dots, \pi_k)_> = ( \dots, \pi_k, \dots, \pi_{p_k}, \dots)
\quad
\text{and}
\quad
(\pi_{p_k}, \dots, \pi_{k-1})_> = ( \dots, \widehat{\pi_k}, \dots, \pi_{p_k}, \dots),
\]
where the hat denotes the absence of the letter. Thus, for any $i$ for which $\pi_i > \pi_k$ at position $j$ in $(\pi_{p_k}, \dots, \pi_k)_>$, the coefficient is
$x_{\pi_i} (q^{\pi_i+j-k-1} - q^{-1} q^{\pi_i+j-(k-1)-1}) = 0$.
Similarly, for any $i$ for which $\pi_i < \pi_k$ at position $j$ in $(\pi_{p_k}, \dots, \pi_k)_>$, the coefficient is
$x_{\pi_i} (q^{\pi_i+j-k-1} - q^{-1} q^{\pi_i+(j-1)-(k-1)-1})$, which is positive. From \eqref{equation.Psi relation}, \cref{lemma.q prime power} and
the fact that the entries of the stationary distribution in \cref{theorem.sd flags} sum to it, it directly follows that $\sum_{\pi \in S_n} \Psi(q,\x)_{\pi} = 1$.
\end{proof}

\subsection{Stationary distribution for words}
\label{ss:ss words}

By~\eqref{eq:permlumping} we have
\[
	 \Psi(q,\xbar)_w = \sum_{\substack{\pi \in \symm_n\\ \proj_{W_\m}(\pi) = w}} \Psi(q,\x)_\pi,
\]
where $\x$ and $\xbar$ are related as in \cref{def:projandinclforwords}.

\begin{lemma}
\label{lemma.x compatible}
When $\x$ is $\m$-compatible, we have
\[
	q^{\inv(\pi)} \Psi(q,\x)_\pi = q^{\inv(\pi')} \Psi(q,\x)_{\pi'}
\]
for all $\pi,\pi'\in \symm_n$ such that $\proj_{W_\m}(\pi)=\proj_{W_\m}(\pi')$.
\end{lemma}

\begin{proof}
The assertion follows from \cref{cor:y equal} and~\eqref{equation.Psi relation}.
\end{proof}

\begin{proof}[Proof of \cref{th:steady state words}]
Note that
\[
	\sum_{\sigma \in \symm_n} q^{\inv(\sigma)} = [n]_q!.
\]
Hence
\[
	\sum_{\substack{\pi \in \symm_n\\ \proj_{W_\m}(\pi) = w}} q^{-\inv(\pi)} = q^{-\inv(w)} \prod_{i=1}^\ell [m_i]_{q^{-1}}!
	= q^{-\inv(w)} \prod_{i=1}^\ell q^{-m_i+1} [m_i]_q! .
\]
Therefore by \cref{lemma.x compatible}, it suffices to show that the expression for $q^{\inv(\pi)} \Psi(q,\x)_\pi$ in \cref{th:eigenvectorH} and
$q^{\inv(w)} \Psi(q,\xbar)_w$ in \cref{th:steady state words} are the same for a given $\pi \in \symm_n$ with 
$\proj_{W_\m}(\pi) = w$. Note that the expressions are exactly the same up to the definitions of $\kappa(\bv;\x)$ and
$\kappa(\bv;\xbar)$. However,
\[
	\kappa(\bv;\xbar) = \sum_{i=1}^k \frac{q^{i+m_1+\cdots+m_{b_i}-k-1}}{[m_{b_i}]_q} \xo_{b_i} = \sum_{i=1}^k q^{n+i-k-1} \yo_{b_i}
\]
since $\yo_j = \frac{\xo_j}{q^{n-n_j}[m_j]_q}$ and
\[
	\kappa(\bv;\x) = \sum_{i=1}^k q^{i+b_i-k-1} x_{b_i} = \sum_{i=1}^k q^{n+i-k-1} y_{b_i}.
\]
Hence, when $\x$ and $\xbar$ are $\m$-compatible, that is, $y_{n_{j-1}+i} = \yo_{n_j}$ for all $1\leqslant i \leqslant m_j$, the first part holds. 

The argument for the second part is identical to that in the proof of 
\cref{th:eigenvectorH}(2) and is omitted. 
The third part follows because the entries
in $\Psi(q,\xbar)_w$ are obtained by summing those in $\Psi(q,\x)_\pi$ using \eqref{eq:permlumping}, which themselves sum to $1$.
\end{proof}

\section{Convergence to stationarity}
\label{section.convergence}

In this section, we prove \cref{cor:relax} and \cref{theorem.mixing} for the mixing time of the $q$-Tsetlin library on flags and permutations.

\begin{proof}[Proof of~\cref{cor:relax}]
Since $d_1 = 0$, there is no eigenvalue $\lambda_S$ with $|S| = n - 1$. With a little bit of work and using $x_i>0$ and $\sum_i x_i=1$, 
one can see that $\lambda_S \leqslant \lambda^*$ for $|S| < n-2$. In the case of $x_i=q^{n-i}/[n]_q$, it is easily seen from the formula 
given in \cref{thm:evalues Sn} that the eigenvalue $\lambda_S$ only depends on the size of $S$ but not on the specific choice of $S$, 
and can be computed as a simple geometric sum. For $x_i=1/n$, the subset maximizing $\lambda_S$ is given by $S=\{n,\ldots,3\}$ 
giving rise to the stated form of $\lambda^*$.
\end{proof}

Recall that a \emph{strong stationary time} for a finite state Markov chain 
is a random time $\tau$ such that the state at time $\tau$
is exactly distributed according to the chain’s stationary distribution regardless of the starting state. It is a standard fact that $\tau$ is a 
stopping time.

\begin{proof}[Proof of \cref{theorem.mixing}]
We will first focus on the $q$-Tsetlin library on flags.
For such Markov chains associated to left regular bands,
Brown and Diaconis~\cite{BD.1998}~\cite[Theorem 0]{Brown} showed that 
the first worst-case time that the walk hits the minimal ideal, denoted $\tau$, is a strong stationary time. 
Further, they showed that $d(t) \leqslant \P[\tau > t]$, where $d(t)$ is defined in \eqref{eq:dtdef}. By Markov's inequality
\[
	d(t) \leqslant \P[\tau > t] \leqslant \frac{\E[\tau]}{t+1},
\]
where $\E[\tau]$ is the expected value for $\tau$.
Let $\E_F$ be the expected time to reach the flag $F$ in the left regular band.
By~\cite[Theorem 2.5]{RhodesSchilling.2022}, we have
\[
\E_{F_\pi} = \sum_{i=1}^n x_i \frac{\partial}{\partial x_i} \ln \Psi(q,\x)_{F_\pi}
\]
if $\Psi(q,\x)_{F_\pi}$ is computed via paths in expansions of the right Cayley graph as in~\cite{RhodesSchilling.2019}---that is,
as in~\eqref{equation.for proof}. We compute
\[
	\E_{F_{\pi}}
	=1 + \sum_{k=1}^{n-1} \frac{1}{1-\sum_{s\in\{\pi_1,\ldots,\pi_k\}} \frac{x_s}{q^{n-s-b_{k+1}(s)}}}.
\]
Evaluating $\E_{F_\pi}$ at $x_i=p^{n-i}/[n]_p$ for $p\geqslant 1$, we obtain
\begin{equation}
\label{equation.E}
	\E_{F_\pi} =1 + \sum_{k=1}^{n-1} \frac{1}{1-\sum_{s\in\{\pi_1,\ldots,\pi_k\}} \frac{p^{n-s}}{[n]_p q^{n-s-b_{k+1}(s)}}},
\end{equation}
and recall that $b_k(s) = \# \{ t\in \{\pi_1,\ldots,\pi_{k-1}\} \mid t>s \}$. 

\medskip

For $q=p>1$, we find from \eqref{equation.E} that
\[
	\E_{F_\pi} =1 + \sum_{k=1}^{n-1} \frac{[n]_q}{[n]_q-\sum_{s\in\{\pi_1,\ldots,\pi_k\}} q^{b_{k+1}(s)}} = \E_{F_{12\ldots n}},
\]
because $\sum_{s\in\{\pi_1,\ldots,\pi_k\}} q^{b_{k+1}(s)}=[k]_q$, which is independent of $\pi$. We thus find
\[
	\E_{F_{12\ldots n}} = 1+ \sum_{k=1}^{n-1} \frac{[n]_q}{[n]_q-[k]_q} = \sum_{k=0}^{n-1} \frac{[n]_q}{[n]_q-[k]_q}
	= \frac{[n]_q}{q^{n-1}} \sum_{k=1}^{n} \frac{1}{[k]_{q^{-1}}} = [n]_{q^{-1}} \sum_{k=1}^{n} \frac{1}{[k]_{q^{-1}}},
\]
where we used $[n]_{q^{-1}} = q^{-n+1} [n]_q$. Furthermore for $q>1$ and $1\leqslant k \leqslant n$, we have
\[
	\frac{1}{1-q^{-n}} \leqslant \frac{1}{1-q^{-k}} \leqslant \frac{1}{1-q^{-1}}
\]
and therefore
\[
	n \leqslant \E_{F_{12\ldots n}} \leqslant n [n]_{q^{-1}} \leqslant n \frac{q}{q-1}.
\]
This proves that $\E[\tau]$ is of order $\Theta(n)$ for $q=p>1$. 

\medskip

For $p,q > 1$, $p\neq q$ we also have that $\E[\tau]$ is of order $\Theta(n)$. 
To prove this we first consider $p>q>1$. In this case the denominator in \eqref{equation.E} for given $k$ is smallest if the powers 
of $p/q$ are largest, which is the case when $s \in \{1,\ldots,k\}$ for which $b_{k+1}(s)=k-s$. It follows that the denominator is smallest for 
$\pi_1\ldots\pi_k=12\ldots k$, and hence
\[
\E_{F_\pi}\leqslant \E_{F_{12\ldots n}}=1+\sum_{k=1}^{n-1} \frac{[n]_p}{[n]_p - (p/q)^{n-k} [k]_p}
= 1+\sum_{k=1}^{n-1} \frac{[n]_{p^{-1}}}{[n]_{p^{-1}} - q^{-n+k} [k]_{p^{-1}}}=\Theta(n). \]
For $q>p>1$ we find
\[
\E_{F_\pi}\leqslant \E_{F_{n(n-1)\ldots 1}}=1+\sum_{k=1}^{n-1} \frac{[n]_p}{[n]_p - [k]_p} =\Theta(n), \]
by the calculation above. 

\medskip

For $q\geqslant 1>p>0$ the denominator in \eqref{equation.E} for given $k$ is smallest when the powers of $p/q$ are smallest, 
which is the case when $s \in \{n-k+1,\ldots,n\}$ for which $b_{k+1}(s)=n-s$. The denominator thus is smallest for $\pi_1\ldots\pi_k=n(n-1)\ldots (n-k+1)$ 
and hence $\E_{F_\pi}\leqslant \E_{F_{n(n-1)\ldots 1}}$. For $p<1$ this is $\Theta(p^{-n})$.

\medskip

For $p>q=1$, we obtain
\begin{align*}
	\E_{F_\pi} &=1 + \sum_{k=1}^{n-1} \frac{1}{1-\sum_{s\in\{\pi_1,\ldots,\pi_k\}} \frac{p^{n-s}}{[n]_p}} \leqslant \E_{F_{12\ldots n}} 
	= 1 + \sum_{k=1}^{n-1} \frac{1}{1-\sum_{s=1}^k \frac{p^{n-s}}{[n]_p}}  \\
	&= \sum_{k=1}^n \frac{[n]_p}{[k]_p} = \sum_{k=1}^n p^{n-k} \frac{[n]_{p^{-1}}}{[k]_{p^{-1}}} \leqslant [n]_p[n]_{p^{-1}} = \Theta(p^n).
\end{align*}
Hence $\E[\tau]$ is of order $\Theta(p^n)$ for $p>q=1$.

\medskip

For $q>p=1$ we obtain
\[
	\E_{F_\pi} =1 + \sum_{k=1}^{n-1} \frac{n}{n-\sum_{s\in\{\pi_1,\ldots,\pi_k\}} q^{-n+s+b_{k+1}(s)} } \leqslant \E_{F_{n(n-1)\ldots 1}}= 1 + \sum_{k=1}^{n-1} \frac{n}{n-k} 
	= \Theta(n\log n).
\]
This bound is sharp and hence $\E[\tau]=\Theta(n\log n)$ for $q>p=1$.

\medskip

For $q=p=1$, we have $x_i=1/n$ and the mixing time is $\Theta(n \log n)$ by~\cite[Theorem 1.1, Remark 1]{diaconisICM}. 
This proves  \cref{theorem.mixing} in the case of flags.

\medskip

As shown in \cref{subsection.diagramconsequences}, the $q$-Tsetlin library on permutations is obtained by lumping or projecting the $q$-Tsetlin library 
on flags. Furthermore, the stationary distributions for all flags associated to the same permutation are equal. Therefore the time taken to reach the 
minimal ideal after projection cannot be larger than $\tau$, proving the result in this case.
\end{proof}

\bibliography{Library}{}
\bibliographystyle{alpha}

\end{document}